\documentclass[a4paper,oneside,english]{amsart}
\usepackage[T1]{fontenc}
\usepackage[utf8]{inputenc}
\usepackage{babel}
\usepackage{mathtools}
\usepackage{amstext}
\usepackage{amsthm}
\usepackage{amssymb}
\usepackage{esint}
\usepackage[unicode=true,pdfusetitle,
 bookmarks=true,bookmarksnumbered=false,bookmarksopen=false,
 breaklinks=false,pdfborder={0 0 1},backref=false,colorlinks=false]
 {hyperref}

\makeatletter

\pdfpageheight\paperheight
\pdfpagewidth\paperwidth

\numberwithin{equation}{section}
\theoremstyle{plain}
\newtheorem{thm}{\protect\theoremname}[section]
\theoremstyle{remark}
\newtheorem{rem}[thm]{\protect\remarkname}
\theoremstyle{plain}
\newtheorem{lem}[thm]{\protect\lemmaname}
\theoremstyle{plain}
\newtheorem{cor}[thm]{\protect\corollaryname}
\theoremstyle{plain}
\newtheorem{prop}[thm]{\protect\propositionname}

\usepackage[initials]{amsrefs}
\usepackage{tikz}

\usepackage{tensor}
\usepackage{enumitem}

\definecolor{green}{rgb}{0,0.8,0} 

\newcommand{\nrm}{\@ifstar{\nrmb}{\nrmi}}
\newcommand{\nrmi}[1]{\Vert{#1}\Vert}
\newcommand{\nrmb}[1]{\left\Vert{#1}\right\Vert}
\newcommand{\abs}{\@ifstar{\absb}{\absi}}
\newcommand{\absi}[1]{\vert{#1}\vert}
\newcommand{\absb}[1]{\left\vert{#1}\right\vert}
\newcommand{\brk}{\@ifstar{\brkb}{\brki}}
\newcommand{\brki}[1]{\langle{#1}\rangle}
\newcommand{\brkb}[1]{\left\langle{#1}\right\rangle}
\newcommand{\set}{\@ifstar{\setb}{\seti}}
\newcommand{\seti}[1]{\{#1\}}
\newcommand{\setb}[1]{\left\{ #1\right\}}

\newcommand{\td}[1]{\widetilde{#1}}
\newcommand{\br}[1]{\overline{#1}}
\newcommand{\ul}[1]{\underline{#1}}
\newcommand{\wh}[1]{\widehat{#1}}

\newcommand{\VERT}[1]{{\left\vert\kern-0.25ex\left\vert\kern-0.25ex\left\vert #1 
    \right\vert\kern-0.25ex\right\vert\kern-0.25ex\right\vert}}

\let\Re\relax
\DeclareMathOperator{\Re}{Re}
\let\Im\relax
\DeclareMathOperator{\Im}{Im}

\newcommand{\rd}{\partial}

\newcommand{\peq}{\relphantom{=}}			

\newcommand{\gmm}{\gamma}

\newcommand{\eps}{\epsilon}

\newcommand{\lmb}{\lambda}

\newcommand{\tht}{\theta}

\newcommand{\bfD}{{\bf D}}

\newcommand{\bbC}{\mathbb C}

\newcommand{\bbN}{\mathbb N}

\newcommand{\bbQ}{\mathbb Q}
\newcommand{\bbR}{\mathbb R}

\newcommand{\bbZ}{\mathbb Z}

\newcommand{\calC}{\mathcal C}

\newcommand{\calE}{\mathcal E}
\newcommand{\calF}{\mathcal F}

\newcommand{\calH}{\mathcal H}

\newcommand{\calL}{\mathcal L}
\newcommand{\calM}{\mathcal M}
\newcommand{\calN}{\mathcal N}

\newcommand{\calZ}{\mathcal Z}

\newcommand{\frkm}{\mathfrak m}

\newcommand{\To}{\longrightarrow}

\newcommand{\weakto}{\rightharpoonup}
\newcommand{\embed}{\hookrightarrow}

\newcommand{\CR}{\bfD_{+}}
\newcommand{\chf}{\boldsymbol{1}}
\newcommand{\tint}[2]{\textstyle \int_{#1}^{#2}}
\newcommand{\tsum}[2]{\textstyle \sum_{#1}^{#2}}

\makeatother

\providecommand{\corollaryname}{Corollary}
\providecommand{\lemmaname}{Lemma}
\providecommand{\propositionname}{Proposition}
\providecommand{\remarkname}{Remark}
\providecommand{\theoremname}{Theorem}

\begin{document}
\global\long\def\bbC{\mathbb{C}}%
\global\long\def\bbN{\mathbb{N}}%
\global\long\def\bbQ{\mathbb{Q}}%
\global\long\def\bbR{\mathbb{R}}%
\global\long\def\bbZ{\mathbb{Z}}%

\global\long\def\bfD{{\bf D}}%

\global\long\def\calC{\mathcal{C}}%
\global\long\def\calE{\mathcal{E}}%
\global\long\def\calF{\mathcal{F}}%
\global\long\def\calH{\mathcal{H}}%
\global\long\def\calL{\mathcal{L}}%
\global\long\def\calM{\mathcal{M}}%
\global\long\def\calN{\mathcal{N}}%
\global\long\def\calZ{\mathcal{Z}}%

\global\long\def\frkm{\mathfrak{m}}%

\global\long\def\eps{\epsilon}%
\global\long\def\lmb{\lambda}%
\global\long\def\gmm{\gamma}%
\global\long\def\tht{\theta}%

\global\long\def\rd{\partial}%

\global\long\def\peq{\mathrel{\phantom{=}}}%
\global\long\def\To{\longrightarrow}%
\global\long\def\weakto{\rightharpoonup}%
\global\long\def\embed{\hookrightarrow}%
\global\long\def\Re{\mathrm{Re}}%
\global\long\def\Im{\mathrm{Im}}%
\global\long\def\chf{\mathbf{1}}%
\global\long\def\td#1{\widetilde{#1}}%
\global\long\def\br#1{\overline{#1}}%
\global\long\def\ul#1{\underline{#1}}%
\global\long\def\wh#1{\widehat{#1}}%
\global\long\def\tint#1#2{{\textstyle \int_{#1}^{#2}}}%
\global\long\def\tsum#1#2{{\textstyle \sum_{#1}^{#2}}}%
\global\long\def\CR{\mathbf{D}_{+}}%

\title[Soliton resolution for CSS]{Soliton resolution for equivariant self-dual Chern--Simons--Schrödinger
equation in weighted Sobolev class}
\author{Kihyun Kim}
\email{khyun@ihes.fr}
\address{IHES, 35 route de Chartres, Bures-sur-Yvette 91440, France}
\author{Soonsik Kwon}
\email{soonsikk@kaist.edu}
\address{Department of Mathematical Sciences, Korea Advanced Institute of Science
and Technology, 291 Daehak-ro, Yuseong-gu, Daejeon 34141, Korea}
\author{Sung-Jin Oh}
\email{sjoh@math.berkeley.edu}
\address{Department of Mathematics, UC Berkeley, Evans Hall 970, Berkeley,
CA 94720-3840, USA and Korea Institue for Advanced Study, 80 Hoegi-ro,
Dongdaemun-gu, Seoul 02455, Korea}
\begin{abstract}
We consider the self-dual Chern--Simons--Schrödinger equation (CSS)
under equivariant symmetry, which is a $L^{2}$-critical equation.
It is known that (CSS) admits solitons and finite-time blow-up solutions.
In this paper, we show \emph{soliton resolution} for any solutions
with equivariant data in the weighted Sobolev space $H^{1,1}$: every
maximal solution decomposes into at most one modulated soliton and
a radiation. A striking fact is that the nonscattering part must be
a single modulated soliton. To our knowledge, this is the first result
on soliton resolution in a class of nonlinear Schrödinger equations
which are not known to be completely integrable. The key ingredient
is the defocusing nature of the equation in the exterior of a soliton
profile. This is a consequence of two distinctive features of (CSS):
self-duality and non-local nonlinearity.
\end{abstract}

\maketitle
\tableofcontents{}

\section{Introduction}

We study the long time dynamics of the self-dual Chern--Simons--Schrödinger
equation \eqref{eq:CSS-m-equiv} under equivariant symmetry. Our main
result (Theorem \ref{thm:asymptotic-description}) is soliton resolution
of solutions in the weighted Sobolev space $H^{1,1}$. Moreover, in
the case of finite-time blow-up, our proof works for all finite energy
solutions.

The \emph{self-dual Chern--Simons--Schrödinger} equation within
\emph{$m$-equivariance} is

\begin{equation}
i(\rd_{t}+iA_{t}[u])u+\rd_{r}^{2}u+\frac{1}{r}\rd_{r}u-\Big(\frac{m+A_{\theta}[u]}{r}\Big)^{2}u+|u|^{2}u=0,\tag{CSS}\label{eq:CSS-m-equiv}
\end{equation}
where $m\in\bbZ$ (called \emph{equivariance index}), and the connection
components $A_{t}[u]$ and $A_{\tht}[u]$ are given by 
\begin{equation}
A_{t}[u]=-\int_{r}^{\infty}(m+A_{\tht}[u])|u|^{2}\frac{dr'}{r'},\qquad A_{\tht}[u]=-\frac{1}{2}\int_{0}^{r}|u|^{2}r'dr'.\label{eq:def-A}
\end{equation}
The Chern--Simons--Schrödinger equation was introduced by Jackiw--Pi
\cite{JackiwPi1990PRL} as a nonrelativistic planar quantum electromagnetic
model that exhibits \emph{self-duality} (to be discussed more below).
It is a gauge-covariant cubic nonlinear Schrödinger equation on $\bbR^{2}$.
We refer to \cite{JackiwPi1990PRL,JackiwPi1990PRD,JackiwPi1991PRD,JackiwPi1992Progr.Theoret.,Dunne1995Springer}
for more physical backgrounds. The model \eqref{eq:CSS-m-equiv} is
derived after fixing the Coulomb gauge condition and imposing the
equivariant symmetry on the scalar field $\phi$: 
\[
\phi(t,x)=u(t,r)e^{im\theta},
\]
where $(r,\theta)$ are the polar coordinates on $\bbR^{2}$. For
more details on this reduction, we refer to the introduction of \cite{KimKwon2019arXiv,KimKwon2020arXiv,KimKwonOh2020arXiv}.

\eqref{eq:CSS-m-equiv} enjoys various symmetries and conservation
laws. Among the most basic symmetries are the time translation and
the phase rotation symmetries. Associated to these are the conservation
laws for the \emph{energy} and the \emph{mass} (the physical interpretation
of the quantity $M[u]$ is the \emph{total charge}, but in this paper
we shall call it mass following the widespread convention for NLS):
\begin{align}
E[u] & \coloneqq\int\frac{1}{2}|\partial_{r}u|^{2}+\frac{1}{2}\Big(\frac{m+A_{\theta}[u]}{r}\Big)^{2}|u|^{2}-\frac{1}{4}|u|^{2},\label{eq:energy-Coulomb-form}\\
M[u] & \coloneqq\int|u|^{2},\label{eq:charge}
\end{align}
where we denoted $\int f(r)=2\pi\int f(r)rdr$. With this energy functional,
\eqref{eq:CSS-m-equiv} admits a \emph{Hamiltonian} structure
\[
\rd_{t}u=-i\nabla E[u],
\]
where $\nabla$ (acting on a functional) is the Fréchet derivative
with respect to the real inner product $\int\Re(\br uv)$. Of particular
importance in this work are the \emph{$L^{2}$-scaling symmetry} and
the \emph{pseudoconformal symmetry};\emph{ }if $u(t,r)$ is a solution
to \eqref{eq:CSS-m-equiv}, then the functions $u_{\lmb}$ and $\calC u$
also solve \eqref{eq:CSS-m-equiv}: 
\begin{align}
u_{\lmb}(t,r) & \coloneqq\frac{1}{\lmb}u\Big(\frac{t}{\lmb^{2}},\frac{r}{\lmb}\Big),\qquad\qquad\forall\lmb>0,\label{eq:scaling}\\{}
[\calC u](t,r) & \coloneqq\frac{1}{|t|}u(-\frac{1}{t},\frac{r}{|t|})e^{\frac{ir^{2}}{4t}},\qquad\forall t\neq0.\label{eq:def-pseudoconf}
\end{align}
Associated to \eqref{eq:scaling} and \eqref{eq:def-pseudoconf} are
the \emph{virial identities}:\emph{ }
\begin{align}
\rd_{t}\int r^{2}|u|^{2} & =4\int\Im(\br u\cdot r\rd_{r}u),\label{eq:virial-1}\\
\rd_{t}\int\Im(\br u\cdot r\rd_{r}u) & =4E[u].\label{eq:virial-2}
\end{align}
In this aspect, \eqref{eq:CSS-m-equiv} shares many similarities with
the cubic NLS 
\begin{equation}
i\partial_{t}\psi+\Delta\psi+|\psi|^{2}\psi=0\quad\text{on }\bbR^{1+2}.\tag{NLS}\label{eq:NLS}
\end{equation}

A notable feature of \eqref{eq:CSS-m-equiv} in comparison to NLS
is the \emph{self-duality}. Indeed, the energy functional can be written
in the self-dual form 
\begin{equation}
E[u]=\int\frac{1}{2}|\bfD_{u}u|^{2},\label{eq:energy-self-dual-form}
\end{equation}
where $\bfD_{u}$ is the (covariant) \emph{Cauchy--Riemann operator}
defined by 
\begin{equation}
\bfD_{u}f\coloneqq\rd_{r}f-\frac{m+A_{\theta}[u]}{r}f.\label{eq:CR-radial}
\end{equation}
We call the operator $u\mapsto\bfD_{u}u$ the \emph{Bogomol'nyi operator}.
Due to \eqref{eq:energy-self-dual-form} and the Hamitonian structure,
any static solutions to \eqref{eq:CSS-m-equiv} are given by solutions
to the \emph{Bogomol'nyi equation}:
\begin{equation}
\bfD_{Q}Q=0.\label{eq:Bogomol'nyi-eq}
\end{equation}
For $m\geq0$, there is an \emph{explicit} $m$-equivariant static
solution (Jackiw--Pi vortex) to the Bogomol'nyi equation which is
unique up to the symmetries of the equation \cite{JackiwPi1990PRD}:
\begin{equation}
Q(r)=\sqrt{8}(m+1)\frac{r^{m}}{1+r^{2m+2}},\qquad m\geq0.\label{eq:Q-formula}
\end{equation}
Note that we suppressed the $m$-dependences in $\bfD_{u}$ and $Q$
for the simplicity of notation. Moreover, applying the pseudoconformal
transform \eqref{eq:def-pseudoconf} to $Q$, we obtain an explicit
finite-time blow-up solution: 
\[
S(t,r)\coloneqq\frac{1}{|t|}Q\Big(\frac{r}{|t|}\Big)e^{-i\frac{r^{2}}{4|t|}},\qquad t<0.
\]
We note that $S(t)$ has finite energy if and only if $m\geq1$.

Let us briefly discuss some known results on the covariant Chern--Simons--Schrödinger
equation without symmetry. The local well-posedness has been studied
by many authors: \cite{BergeDeBouardSaut1995Nonlinearity,Huh2013Abstr.Appl.Anal,LiuSmithTataru2014IMRN,Lim2018JDE}.
However, the best known result by Liu--Smith--Tataru \cite{LiuSmithTataru2014IMRN}
still misses the critical $L^{2}$-space. There are also results on
the long-term dynamics \cite{BergeDeBouardSaut1995Nonlinearity,BergeDeBouardSaut1995PRL,OhPusateri2015}.

If one restricts to the equivariant self-dual Chern--Simons--Schrödinger
equation, i.e., \eqref{eq:CSS-m-equiv}, then much more is known.
First of all, as there is no derivative nonlinearity, \eqref{eq:CSS-m-equiv}
is well-posed in $L^{2}$ \cite[Section 2]{LiuSmith2016}. The global-in-time
large data dynamics are partially known. Here, the ground state $Q$
provides a natural threshold for the nonscattering dynamics. Indeed,
Liu--Smith \cite{LiuSmith2016} proved the following \emph{subthreshold
theorem}: for $m\geq0$, any $m$-equivariant $L^{2}$-solutions $u$
with $M[u]<M[Q]$ scatter both forwards and backwards in time. At
the threshold mass $M[u]=M[Q]$ (necessarily $m\geq0$), there are
two typical examples of non-scattering solutions: $Q$ and $S(t)$.
These are indeed the only examples in the energy space due to the
classification result of Li--Liu \cite{LiLiu2020arXiv} ($S(t)$
for the radial case $m=0$ is an exception because it does not have
finite energy due to the slow spatial decay of $Q$). Above the threshold,
 \cite{KimKwon2019arXiv,KimKwon2020arXiv,KimKwonOh2020arXiv,KimKwonOh2022arXiv2}
provide a variety of finite-time blow-up solutions (and global-in-time
nonscattering solutions) with quantitative descriptions of the dynamics
near the blow-up time.

It is widely believed that, for arbitrary large data, the maximal
solutions asymptotically decompose into the sum of decoupled solitons
and a radiation. This is referred to as the \emph{soliton resolution
conjecture}. This has been known for a wide range of completely integrable
equations, but the focus of the present paper is on soliton resolution
for (possibly)\emph{ non-integrable} models without exploiting complete
integrability techniques. Recently, the remarkable works \cite{DuyckaertsKenigMerle2013CambJMath,DuyckaertsKenigMere2019arXiv1,DuyckaertsKenigMartelMerle2021arXiv,CollotDuyckaertsKenigMere2022arXiv,JendrejLawrie2021arXiv}
established soliton resolution for the radial critical nonlinear wave
equation (in various dimensions) and energy-critical equivariant wave
maps. However, to our knowledge, there is no earlier result for non-integrable
Schrödinger type equations.

The main result of this paper\emph{ }is\emph{ the proof of soliton
resolution for the equivariant self-dual Chern--Simons--Schrödinger
equation in a suitable weighted Sobolev class}. Our proof is based
on a remarkable consequence of the non-local nonlinearity and the
self-duality of \eqref{eq:CSS-m-equiv}, namely, the defocusing nature
of the equation in the exterior of a soliton profile. This property
also results in a strong rigidity of the dynamics of \eqref{eq:CSS-m-equiv}:
the non-existence of multi-soliton configurations separated by scales.
See the remarks following Theorem \ref{thm:asymptotic-description}.

We are now ready to state the result. Let us denote the modulated
soliton by 
\[
Q_{\lambda,\gamma}(r)\coloneqq\frac{e^{i\gamma}}{\lambda}Q\Big(\frac{r}{\lambda}\Big),\qquad\lmb\in(0,\infty),\ \gmm\in\bbR/2\pi\bbZ.
\]
We also denote by $H_{m}^{1,1}$ and $H_{m}^{1}$ the (weighted) Sobolev
spaces $H^{1,1}$ and $H^{1}$ restricted to $m$-equivariant functions,
equipped with the inherited norms. We denote by $\Delta^{(m)}=\rd_{rr}+\tfrac{1}{r}\rd_{r}-\tfrac{m^{2}}{r^{2}}$
the Laplacian acting on $m$-equivariant functions.
\begin{thm}[Soliton resolution for equivariant $H^{1,1}$-data]
\label{thm:asymptotic-description}Let $m\in\bbZ$. When $m\geq0$,
we have soliton resolution for $H_{m}^{1,1}$-solutions:
\begin{itemize}
\item (Finite-time blow-up solutions) If $u$ is a $H_{m}^{1}$-solution
to \eqref{eq:CSS-m-equiv} that blows up forwards in time at $T<+\infty$,
then $u(t)$ admits the decomposition 
\begin{equation}
u(t,\cdot)-Q_{\lambda(t),\gamma(t)}\to z^{\ast}\text{ in }L^{2}\text{ as }t\to T^{-},\label{eq:thm1-decomp}
\end{equation}
for some continuous $\lambda(t)\in(0,\infty)$ and $\gamma(t)\in\bbR/2\pi\bbZ$,
and $z^{\ast}\in L^{2}$ with the following properties:
\begin{itemize}
\item (Further regularity of $z^{\ast}$) We have $\partial_{r}z^{\ast},\frac{1}{r}z^{\ast}\in L^{2}$.
Moreover, if $u$ is a $H_{m}^{1,1}$ finite-time blow-up solution,
then we also have $rz^{\ast}\in L^{2}$.
\item (Bound on the blow-up speed) As $t\to T$, we have 
\begin{equation}
\lambda(t)\lesssim_{M[u]}\sqrt{E[u]}(T-t).\label{eq:thm1-lmb-upper-bound}
\end{equation}
When $m=0$, we further have the improved bound as $t\to T$
\begin{equation}
\lambda(t)\lesssim_{M[u]}\frac{\sqrt{E[u]}(T-t)}{|\log(T-t)|^{\frac{1}{2}}}.\label{eq:thm1-lmb-upper-bound-radial}
\end{equation}
\end{itemize}
\item (Global solutions) If $u$ is a $H_{m}^{1,1}$-solution to \eqref{eq:CSS-m-equiv}
that exists globally forwards in time, then either $u(t)$ scatters
forwards in time, or $u(t)$ admits the decomposition 
\begin{equation}
u(t,\cdot)-Q_{\lambda(t),\gamma(t)}-e^{it\Delta^{(-m-2)}}u^{\ast}\to0\text{ in }L^{2}\text{ as }t\to+\infty,\label{eq:thm-global-decomp}
\end{equation}
for some continuous $\lambda(t)\in(0,\infty)$ and $\gamma(t)\in\bbR/2\pi\bbZ$,
and $u^{\ast}\in L^{2}$ with the following properties:
\begin{itemize}
\item (Further regularity of $u^{\ast}$) We have $\partial_{r}u^{\ast},\frac{1}{r}u^{\ast},ru^{\ast}\in L^{2}$.
\item (Bound on the scale) As $t\to+\infty$, we have 
\begin{equation}
\lmb(t)\lesssim_{M[u]}\sqrt{E[\calC u]},\label{eq:thm1-lmb-global-bound}
\end{equation}
where $\calC u$ is the pseudoconformal transform \eqref{eq:def-pseudoconf}
of $u$. When $m=0$, we further have as $t\to+\infty$
\begin{equation}
\lmb(t)\lesssim_{M[u]}\frac{\sqrt{E[\calC u]}}{|\log t|^{\frac{1}{2}}}.\label{eq:thm1-lmb-global-bound-radial}
\end{equation}
\end{itemize}
\end{itemize}
On the other hand, when $m<0$, any $H_{m}^{1,1}$-solution to \eqref{eq:CSS-m-equiv}
scatters forwards in time. Due to the time-reversal symmetry, all
the above statements also hold for backward-in-time evolutions.
\end{thm}

We note that one can further choose smooth modulation parameters in
Theorem~\ref{thm:asymptotic-description} because the theorem is
invariant under replacing $\lmb(t)$ by any function $\td{\lmb}(t)$
with $\td{\lmb}(t)/\lmb(t)\to1$ (and similarly for $\gmm(t)$).
\begin{rem}[The dynamics for $m\geq0$ and $m<0$]
The dynamics of \eqref{eq:CSS-m-equiv} for $m\geq0$ and $m<0$
are completely different. In fact, we will show that \eqref{eq:CSS-m-equiv}
for $m<0$ is \emph{defocusing} in the sense that 
\begin{equation}
E[u]\sim_{M[u]}\|u(t)\|_{\dot{H}_{m}^{1}}^{2}\label{eq:intro-defocusing-neg-equiv}
\end{equation}
and hence there are no nontrivial Jackiw--Pi vortices for $m<0$.
See Lemma \ref{lem:Nonlin-coer-neg-equiv} for the proof.
\end{rem}

\begin{rem}[Nonexistence of multi-solitons]
\label{rem:Nonexistence-multisoliton}It is remarkable that at most
one soliton can appear in the resolution. This is a distinctive feature
of \eqref{eq:CSS-m-equiv}. Indeed, as a consequence of the self-duality
and non-locality, we observe a \emph{defocusing }nature, i.e., the
strict positivity of the energy, of \eqref{eq:CSS-m-equiv} at the
exterior of a soliton profile. Hence two solitons at different scales
cannot exist simultaneously. We will obtain this defocusing nature
by combining our two observations: (i) \eqref{eq:CSS-m-equiv} at
the exterior of soliton resembles \eqref{eq:CSS-m-equiv} for $m<0$
(observed in \cite{KimKwon2019arXiv}) and (ii) the defocusing nature
\eqref{eq:intro-defocusing-neg-equiv} when $m<0$. See Lemma \ref{lem:Nonlin-coer-neg-equiv}
for the proof.

Even without equivariant symmetry, by essentially the same mechanism,
we expect that there is no bubble tree (i.e., a multi-soliton separated
\emph{only} by scales) for the self-dual Chern--Simons--Schrödinger
equation. However, multi-solitons separated by spatial distances may
exist.
\end{rem}

\begin{rem}[Regularity assumptions on data]
As seen in the above, we cover all $H_{m}^{1}$ finite-time blow-up
solutions. For global solutions, we will reduce the situation to the
$H_{m}^{1}$ finite-time blow-up case in the spirit of the pseudoconformal
transform, which requires the $H_{m}^{1,1}$-assumption. Note that
$E[\calC u]$ is well-defined for $H_{m}^{1,1}$-solutions $u$. Soliton
resolution for any global $H_{m}^{1}$-solutions (or, more ambitiously
$L_{m}^{2}$-solutions) is an interesting open problem.
\end{rem}

\begin{rem}[Equivariance index on the scattering part]
\label{rem:scattering-diff-equiv}We choose to state the scattering
for the radiative part of \eqref{eq:thm-global-decomp} under the
$(-m-2)$-equivariant free Schrödinger flow, because the scattering
part $u(t)-Q_{\lmb(t),\gmm(t)}$ approximately solves $(-m-2)$-equivariant
\eqref{eq:CSS-m-equiv}. This fact is already observed in \cite{KimKwon2019arXiv}.

However, the equivariance index for the scattering part of \eqref{eq:thm-global-decomp}
is irrelevant if one only considers the scattering in $L^{2}$-norm.
Indeed, for any $m,k\in\bbZ$ and radial functions $u^{\ast},v^{\ast}\in L^{2}$
(we equip $L^{2}$ with the $rdr$-measure), we have 
\[
\|e^{it\Delta^{(m)}}u^{\ast}-e^{it\Delta^{(k)}}v^{\ast}\|_{L^{2}}\to0\qquad\text{as }t\to+\infty
\]
if $u^{\ast}$ and $v^{\ast}$ satisfy the relation
\[
u^{\ast}=\calF_{m}^{-1}\calF_{k}v^{\ast},
\]
where $\calF_{m}$ is the rescaled version of the Hankel transform
of order $m$: 
\[
[\calF_{m}f](\rho)=\frac{i^{m-1}}{2}\int_{0}^{\infty}f(r)J_{m}(\frac{r\rho}{2})rdr
\]
with $J_{m}$ Bessel function of the first kind of order $m$. Note
that the above can be verified using the pseudoconformal transform
and the identity $e^{it\Delta^{(m)}}u^{\ast}=[\mathcal{C}e^{i(\cdot)\Delta^{(m)}}\calF_{m}u^{\ast}](t)$
for $t>0$, for any $m\in\bbZ$. Since $\calF_{m}^{-1}\calF_{k}$
is unitary in $L^{2}$, the $L^{2}$-scattering is independent of
equivariance indices. However, the scattering with different equivariance
indices might not be equivalent under topologies other than $L^{2}$.
For example, the properties $\rd_{r}u^{\ast},\frac{1}{r}u^{\ast},ru^{\ast}\in L^{2}$
(as stated in Theorem \ref{thm:asymptotic-description}) may not be
preserved under changes of equivariance indices, i.e., under $\calF_{m}^{-1}\calF_{k}$.
\end{rem}

\begin{rem}[Bounds for scaling parameter]
\label{rem:rem1.4}When $m\geq1$, the explicit blow-up solution
$S(t)$ and the pseudoconformal blow-up solutions constructed in \cite{KimKwon2019arXiv,KimKwon2020arXiv}
are finite-energy finite-time blow-up solutions that saturate the
bound \eqref{eq:thm1-lmb-upper-bound}. Similarly, the soliton $Q$
itself saturates \eqref{eq:thm1-lmb-global-bound}.

When $m=0$, the blow-up solution $S(t)$ and the soliton $Q$ do
not satisfy the bounds \eqref{eq:thm1-lmb-upper-bound-radial} and
\eqref{eq:thm1-lmb-global-bound-radial}, respectively. This is consistent
with Theorem \ref{thm:asymptotic-description} because $Q$ does not
belong to $H_{0}^{1,1}$ and the explicit blow-up solution $S(t)$
does not have finite energy, and hence $Q$ and $S(t)$ are not covered
by our theorem. Note that \eqref{eq:thm1-lmb-global-bound-radial}
says that any global-in-time nonscattering $H_{0}^{1,1}$-solution
must blow up in infinite time. On the other hand, the authors \cite{KimKwonOh2020arXiv,KimKwonOh2022arXiv2}
construct finite energy finite-time blow-up solutions with the speed
$\lambda(t)\sim(T-t)|\log(T-t)|^{-2}$ and $\lmb(t)\sim(T-t)^{p}|\log(T-t)|^{-1}$
for all $p>1$, respectively. However, we do not know whether the
upper bound \eqref{eq:thm1-lmb-upper-bound-radial} is sharp or not.

The bounds \eqref{eq:thm1-lmb-upper-bound}-\eqref{eq:thm1-lmb-upper-bound-radial}
and \eqref{eq:thm1-lmb-global-bound}-\eqref{eq:thm1-lmb-global-bound-radial}
may have to be relaxed for more general class of solutions. Note that
such a relaxation is necessary for $m=0$ by the concrete examples
$S(t)$ and $Q$.
\end{rem}

\begin{rem}[On the phase rotation parameter]
The phase rotation parameter does not necessarily stabilize as $t\to T$
(or $t\to+\infty$). Indeed, the finite-time blow-up solutions constructed
in \cite{KimKwonOh2022arXiv2} for the $m=0$ case exhibit infinite
amount of phase rotations. The $m\geq1$ case is open.
\end{rem}

\begin{rem}[Comparison with \eqref{eq:NLS}]
\label{rem:ComparisonNLS}For the finite-time blow-up case, there
are similar results \cite{Raphael2005MathAnn,MerleRaphael2005CMP}
in \eqref{eq:NLS} for solutions having slightly supercritical mass
(i.e., $M[u]-M[Q]\ll1$). Under this assumption, a standard variational
argument in the blow-up scenario ensures that solutions eventually
undergo the near-soliton dynamics in the \emph{$L^{2}$-topology}.
Note that in Theorem \ref{thm:asymptotic-description} we do not have
$L^{2}$-proximity to solitons.

For the near-soliton dynamics of \eqref{eq:NLS}, it is known from
\cite{Raphael2005MathAnn} that any finite energy finite-time blow-up
solutions satisfy either $\lambda(t)\sim((T-t)/\log|\log(T-t)|)^{\frac{1}{2}}$
or $\lambda(t)\lesssim(T-t)$. The former log-log rate essentially
arises from negative energy solutions, which is impossible for the
self-dual (CSS). It is expected that such log-log rates are possible
for the focusing non-self-dual (CSS) \cite{BergeDeBouardSaut1995PRL}.
\end{rem}

\noindent \mbox{\textbf{Strategy of the proof.} }We use the notation
in Section \ref{subsec:Notation}.

For $m<0$, we will prove the global coercivity of energy \eqref{eq:intro-defocusing-neg-equiv},
which renders \eqref{eq:CSS-m-equiv} essentially defocusing for $m<0$.
Thus the scattering for $H_{m}^{1,1}$-data follows from a classical
argument using the pseudoconformal transform, see e.g., \cite{Cazenave2003book}.

The interesting case is when $m\geq0$, where solitons do exist. By
the pseudoconformal transform, it suffices to prove the finite-time
blow-up case of Theorem \ref{thm:asymptotic-description}. \emph{Our
key input is the nonlinear coercivity of energy \eqref{eq:intro-nonlin-coer}
after extracting out the soliton profile, which holds for solutions
with possibly large mass}. As explained in Remark \ref{rem:Nonexistence-multisoliton},
this nonlinear coercivity is a consequence of the \emph{self-duality}
and \emph{non-locality}, which are distinctive features of \eqref{eq:CSS-m-equiv}.

\emph{1. Variational argument.} For a finite energy finite-time blow-up
solution $u(t)$, not necessarily close to the modulated soliton $Q_{\lmb,\gmm}$
in $L^{2}$, we work with the renormalized solution $v(t)$ (using
the $L^{2}$-scaling) near the blow-up time (say $T$) of $u$ to
have $\|v(t)\|_{\dot{H}_{m}^{1}}=\|Q\|_{\dot{H}_{m}^{1}}$ and $E[v(t)]\to0$.

In view of the \emph{uniqueness} of the zero energy solution (i.e.,
$E[w]=0$ if and only if $w=Q_{\lmb,\gmm}$ or $w=0$) and renormalization,
we expect that each $v(t)$ is close to $e^{i\gmm(t)}Q$. We remark
that the closeness of $v$ to $e^{i\gmm}Q$ cannot be measured in
$L^{2}$, because we do not assume that the mass of $v$ is close
to that of $Q$. In fact, we are able to show that $v$ is close to
$e^{i\gmm}Q$ in the $\dot{H}^{1}$-topology (Lemma \ref{lem:orbital-stability-small-energy}).
Therefore, we roughly have 
\begin{equation}
u(t)=[Q+\eps(t,\cdot)]_{\lmb(t),\gmm(t)}\quad\text{with}\quad\|\eps(t)\|_{\dot{H}^{1}}\to0\label{eq:intro-decomp}
\end{equation}
for some $\lmb(t)$ and $\gmm(t)$. We may fix the decomposition by
imposing suitable orthogonality conditions on $\eps$. We note that
$\|\eps(t)\|_{L^{2}}$ might be large. We also note that \eqref{eq:intro-decomp}
is a consequence of the uniqueness of zero-energy solutions to \eqref{eq:CSS-m-equiv};
one cannot expect \eqref{eq:intro-decomp} for \eqref{eq:NLS} for
arbitrary solutions with large mass.

\emph{2. Nonlinear coercivity of energy.} For the proof of Theorem
\ref{thm:asymptotic-description}, the qualitative information $\|\eps(t)\|_{\dot{H}^{1}}\to0$
is not sufficient. Our next crucial input is the following \emph{nonlinear
coercivity of the energy} (Lemma \ref{lem:nonlinear-coercivity}):
\begin{equation}
E[Q+\eps]\gtrsim_{\|\eps\|_{L^{2}}}\|\eps\|_{\dot{H}^{1}}^{2}\label{eq:intro-nonlin-coer}
\end{equation}
for $\eps$ satisfying the orthogonality conditions and $\|\eps\|_{\dot{H}^{1}}\ll1$.
Here, the point is that \emph{the coercivity holds even for $\|\eps(t)\|_{L^{2}}\gtrsim1$}.
If we were to have $L^{2}$-smallness $\|\eps(t)\|_{L^{2}}\ll1$,
then all the higher order terms of $E[Q+\eps]$ are perturbative and
\eqref{eq:intro-nonlin-coer} is merely a consequence of the linear
coercivity (around $Q$). When $\eps(t)$ has large $L^{2}$-norm,
the higher order terms of $E[Q+\eps]$ are no longer perturbative.
Instead, we have (using the \emph{self-duality} \eqref{eq:energy-self-dual-form})
\begin{align*}
E[Q+\eps] & =\frac{1}{2}\int|\bfD_{Q+\eps}(Q+\eps)|^{2}\\
 & \approx\frac{1}{2}\int|L_{Q}(\chi_{R}\eps)|^{2}+\Big|\Big(\rd_{r}-\frac{m+A_{\theta}[Q]+A_{\theta}[\eps]}{r}\Big)\big((1-\chi_{R})\eps\big)\Big|^{2},
\end{align*}
where $L_{Q}$ is the linearized Bogomol'nyi operator around $Q$
(see \eqref{eq:LinearizationBogomolnyi}). The interior term is simply
handled by a localized version of the linear coercivity for $L_{Q}$.
However, the exterior term contains non-perturbative higher order
terms like $|\frac{A_{\theta}[\eps]}{r}\eps|^{2}$. At this point,
we use the \emph{non-locality} of the problem, particularly the fact
that $m+A_{\theta}[Q]\approx-(m+2)$ is negative. Thus the exterior
term can be viewed as the energy of $\eps$ for the $-(m+2)$-equivariant
\eqref{eq:CSS-m-equiv}. Using the boundary condition $[(1-\chi_{R})\eps](R)=0$
and the fact that both $m+A_{\theta}[Q]$ and $A_{\theta}[\eps]$
are negative, we can prove unconditional coercivity \eqref{eq:HardyClaim}
for the exterior term (which we call nonlinear Hardy's inequality).
Note that this argument also shows the nonexistence of nontrivial
zero energy solutions to \eqref{eq:CSS-m-equiv} for negative equivariance
indices. As a result, the nonlinear coercivity of energy \eqref{eq:intro-nonlin-coer}
follows.

\emph{3. Bound on the blow-up rate. }The proof of \eqref{eq:thm1-lmb-upper-bound}
is standard and very similar to the pseudoconformal regime in Raphaël
\cite{Raphael2005MathAnn}\emph{. }Indeed, by a standard modulation
analysis, we obtain a modulation estimate in the renormalized spacetime
variables \eqref{eq:def-renormalized-var}: 
\[
\Big|\frac{\lmb_{s}}{\lmb}\Big|\lesssim\|\eps\|_{\dot{H}_{m}^{1}}.
\]
Then, thanks to the nonlinear coercivity \eqref{eq:intro-nonlin-coer}
of energy, we get 
\[
|\lmb_{t}|=\frac{1}{\lmb}\Big|\frac{\lmb_{s}}{\lmb}\Big|\lesssim\frac{1}{\lmb}\sqrt{E[Q+\eps]}=\sqrt{E[u]},
\]
whose integration yields the bound \eqref{eq:thm1-lmb-upper-bound}.

However, the proof of \eqref{eq:thm1-lmb-upper-bound-radial} for
$m=0$ is trickier. We will use the fact that $yQ\notin L^{2}$ (logarithmic
divergence). Indeed, motivated by the generalized nullspace relations
\eqref{eq:gen-kernel-rel} of the linearized operator $i\calL_{Q}$
(see \eqref{eq:grad-energy-lin-nonlin}), the time variation of $\lmb$
can be tracked by looking at the time evolution of the inner product
$(\eps,y^{2}Q)_{r}$; we roughly have an estimate of the form 
\[
\frac{\lmb_{s}}{\lmb}(\Lambda Q,y^{2}Q\chi_{R})_{r}-\rd_{s}(\eps,y^{2}Q\chi_{R})_{r}\approx-(i\calL_{Q}\eps,y^{2}Q\chi_{R})_{r}\lesssim\|\eps\|_{\dot{H}_{m}^{1}}\|yQ\chi_{R}\|_{L^{2}}.
\]
The point here is that we have different logarithmic divergences of
the quantities: 
\begin{align*}
(\Lambda Q,y^{2}Q\chi_{R})_{r} & \sim\log R,\\
\|yQ\chi_{R}\|_{L^{2}} & \sim\sqrt{\log R}.
\end{align*}
Next, we choose $R=R(t)$ which diverges polynomially so that $\log R\sim|\log(T-t)|$
and absorb $\rd_{s}(\eps,y^{2}Q\chi_{R})_{r}$ into a total derivative
to roughly have 
\[
\Big|\frac{\lmb_{s}}{\lmb}\Big|\lesssim\frac{1}{|\log(T-t)|^{\frac{1}{2}}}\|\eps\|_{\dot{H}^{1}}.
\]
Using \eqref{eq:intro-nonlin-coer} again, this implies \eqref{eq:thm1-lmb-upper-bound-radial}.

\emph{4. Existence of the asymptotic profile.} The existence of the
asymptotic profile $z^{\ast}$ as in \eqref{eq:thm1-decomp} as well
as its regularity can be proved in a very similar manner to Merle--Raphaël
\cite{MerleRaphael2005CMP}. To obtain $z^{\ast}$ as the strong $L^{2}$-limit
of $\eps^{\sharp}(t)$ as $t\to T^{-}$, we again take advantage of
the nonlinear coercivity of energy in the form $\|\eps^{\sharp}(t)\|_{H_{m}^{1}}\lesssim1$.
This means that $\eps^{\sharp}(t)$ (and hence $z^{\ast}$) is not
only controlled on the obvious soliton scale $r\lesssim\lmb$, but
also up to scale $r\lesssim1$.\vspace{5bp}

\noindent \mbox{\textbf{Organization of the paper.} }In Section
\ref{sec:Preliminaries}, we collect notation and preliminaries for
our analysis. In Section \ref{sec:Proof-of-Theorem-neg-equiv}, we
prove Theorem \ref{thm:asymptotic-description} for $m<0$. The heart
of this paper is contained in Section \ref{sec:Proof-of-Theorem-pos-equiv},
where we prove Theorem \ref{thm:asymptotic-description} for $m\geq0$.

\vspace{5bp}

\noindent \mbox{\textbf{Acknowledgements.} }K.~Kim is supported
by  Huawei Young Talents Programme at IHES. S.~Kwon is partially
supported by Samsung Science \& Technology Foundation under Project
Number BA1701-01, NRF-2019R1A5A1028324, and NRF-2018R1D1A1A0908335.
S.-J.~Oh is supported by the Samsung Science \& Technology Foundation
under Project Number BA1702-02, a Sloan Research Fellowship and a
NSF CAREER Grant DMS-1945615.

\section{\label{sec:Preliminaries}Preliminaries}

In this section, we collect notation and preliminary facts on linearization,
adapted function spaces, and duality estimates for \eqref{eq:CSS-m-equiv}.

\subsection{\label{subsec:Notation}Notation}

For $A\in\bbC$ and $B\geq0$, we use the standard asymptotic notation
$A\lesssim B$ or $A=O(B)$ if there is a constant $C>0$ such that
$|A|\leq CB$. $A\sim B$ means that $A\lesssim B$ and $B\lesssim A$.
The dependencies of $C$ are specified by subscripts, e.g., $A\lesssim_{E}B\Leftrightarrow A=O_{E}(B)\Leftrightarrow|A|\leq C(E)B$.
In this paper, any dependencies on the equivariance index $m$ will
be omitted.

We also use the notation $\langle x\rangle$, $\log_{+}x$, $\log_{-}x$
defined by 
\[
\langle x\rangle\coloneqq(|x|^{2}+1)^{\frac{1}{2}},\quad\log_{+}x\coloneqq\max\{\log x,0\},\quad\log_{-}x\coloneqq\max\{-\log x,0\}.
\]
We let $\chi=\chi(x)$ be a smooth spherically symmetric cutoff function
such that $\chi(x)=1$ for $|x|\leq1$ and $\chi(x)=0$ for $|x|\geq2$.
For $A>0$, we define its rescaled version by $\chi_{A}(x)\coloneqq\chi(x/A)$.

We mainly work with equivariant functions on $\bbR^{2}$, say $\phi:\bbR^{2}\to\bbC$,
or equivalently their radial part $u:\bbR_{+}\to\bbC$ with $\phi(x)=u(r)e^{im\theta}$,
where $\bbR_{+}\coloneqq(0,\infty)$ and $x_{1}+ix_{2}=re^{i\theta}$.
We denote by $\Delta^{(m)}=\rd_{rr}+\tfrac{1}{r}\rd_{r}-\tfrac{m^{2}}{r^{2}}$
the Laplacian acting on $m$-equivariant functions.

The integral symbol $\int$ means 
\[
\int=\int_{\bbR^{2}}\,dx=2\pi\int\,rdr.
\]
For complex-valued functions $f$ and $g$, we define their \emph{real}
inner product by 
\[
(f,g)_{r}\coloneqq\int\Re(\br fg).
\]
For a real-valued functional $F$ and a function $u$, we denote by
$\nabla F[u]$ the \emph{functional derivative} of $F$ at $u$ under
this real inner product.

We denote by $\Lambda$ the $L^{2}$-scaling generator: 
\[
\Lambda\coloneqq r\rd_{r}+1.
\]
Given a scaling parameter $\lmb\in\bbR_{+}$, phase rotation parameter
$\gmm\in\bbR/2\pi\bbZ$, and a function $f$, we write 
\[
f_{\lmb,\gmm}(r)\coloneqq\frac{e^{i\gmm}}{\lmb}f\Big(\frac{r}{\lmb}\Big).
\]
When $\lmb$ and $\gmm$ are clear from the context, we will also
denote the above by $f^{\sharp}$ as in \cite{KimKwon2019arXiv},
i.e., 
\[
f^{\sharp}\coloneqq f_{\lmb,\gmm}.
\]
Similarly, we define its inverse by $\flat$: 
\[
g^{\flat}(y)\coloneqq\lmb e^{-i\gmm}g(\lmb y).
\]
When a time-dependent scaling parameter $\lmb(t)$ is given, we define
rescaled spacetime variables $s,y$ by the relations 
\begin{equation}
\frac{ds}{dt}=\frac{1}{\lmb^{2}(t)}\qquad\text{and}\qquad y=\frac{r}{\lmb(t)}.\label{eq:def-renormalized-var}
\end{equation}
The raising operation $\sharp$ converts a function $f=f(y)$ to a
function of $r$: $f^{\sharp}=f^{\sharp}(r)$. Similarly, the lowering
operation $\flat$ converts a function $g=g(r)$ to a function of
$y$: $g^{\flat}=g^{\flat}(y)$. In the modulation analysis in this
paper, the dynamical parameters such as $\lmb,\gmm,b,\eta$ are functions
of either the variable $t$ or $s$ under $\frac{ds}{dt}=\frac{1}{\lmb^{2}}$.

For $k\in\bbN$, we define 
\begin{align*}
|f|_{k} & \coloneqq\max\{|f|,|r\rd_{r}f|,\dots,|(r\rd_{r})^{k}f|\},\\
|f|_{-k} & \coloneqq\max\{|\rd_{r}^{k}f|,|\tfrac{1}{r}\rd_{r}^{k-1}f|,\dots,|\tfrac{1}{r^{k}}f|\}.
\end{align*}
We note that $|f|_{k}\sim r^{k}|f|_{-k}$. The following Leibniz rules
hold: 
\[
|fg|_{k}\lesssim|f|_{k}|g|_{k},\qquad|fg|_{-k}\lesssim|f|_{k}|g|_{-k}.
\]

The relevant function spaces will be discussed in Section \ref{subsec:Adapted-function-spaces}.

\subsection{\label{subsec:Linearization-of-CSS}Linearization of \eqref{eq:CSS-m-equiv}}

We quickly record the linearization of \eqref{eq:CSS-m-equiv} around
$Q$. For more detailed exposition, see the corresponding sections
of \cite{KimKwon2019arXiv,KimKwon2020arXiv,KimKwonOh2020arXiv}.

We first linearize the Bogomol'nyi operator $w\mapsto\bfD_{w}w$.
We can write 
\begin{equation}
\bfD_{w+\eps}(w+\eps)=\bfD_{w}w+L_{w}\eps+\text{(h.o.t)},\label{eq:LinearizationBogomolnyi}
\end{equation}
where 
\begin{align*}
L_{w}\eps & \coloneqq\bfD_{w}\eps-\tfrac{2}{y}A_{\tht}[w,\eps]w,
\end{align*}
and $A_{\theta}[\psi_{1},\psi_{2}]$ is defined through the polarization
\[
A_{\tht}[\psi_{1},\psi_{2}]\coloneqq-\tfrac{1}{2}\tint 0r\Re(\br{\psi_{1}}\psi_{2})r'dr'.
\]
The $L^{2}$-adjoint $L_{w}^{\ast}$ of $L_{w}$ takes the form 
\[
L_{w}^{\ast}v=\bfD_{w}^{\ast}v+w{\textstyle \int_{y}^{\infty}}\Re(\br wv)\,dy'.
\]
We remark that the operator $L_{w}$ and its adjoint $L_{w}^{\ast}$
are \emph{only} $\bbR$-linear. From $\bfD_{Q}Q=0$ and \eqref{eq:energy-self-dual-form},
we have the following expansion for the energy: 
\begin{equation}
E[Q+\eps]=\tfrac{1}{2}\|L_{Q}\eps\|_{L^{2}}^{2}+\text{(h.o.t.)}.\label{eq:linearized-energy-expn}
\end{equation}

Next, we linearize \eqref{eq:CSS-m-equiv}, which we write in the
Hamiltonian form $\rd_{t}u+i\nabla E[u]=0$. We decompose 
\begin{equation}
\nabla E[w+\eps]=\nabla E[w]+\mathcal{L}_{w}\eps+R_{w}(\eps),\label{eq:grad-energy-lin-nonlin}
\end{equation}
where $\calL_{w}\eps$ collects the linear terms in $\eps$ and $R_{w}(\eps)$
collects the remainders. Note that $\calL_{w}$ is the Hessian of
$E$, i.e., 
\[
\nabla^{2}E[w]=\calL_{w}.
\]
Being the Hessian of the energy, $\calL_{w}$ is formally symmetric
with respect to the real inner product: 
\[
(\calL_{w}f,g)_{r}=(f,\calL_{w}g)_{r}.
\]
If one recalls \eqref{eq:energy-self-dual-form}, we have $\nabla E[u]=L_{u}^{\ast}\bfD_{u}u$.
Thus \eqref{eq:grad-energy-lin-nonlin} and \eqref{eq:LinearizationBogomolnyi}
yield 
\begin{align*}
\calL_{w}\eps & =L_{w}^{\ast}L_{w}\eps+(\tfrac{1}{y}\tint 0y\Re(\br w\eps)y'dy')\bfD_{w}w\\
 & \qquad\qquad\qquad+w\tint y{\infty}\Re(\br{\eps}\bfD_{w}w)dy'+\eps\tint y{\infty}\Re(\br w\bfD_{w}w)dy'.
\end{align*}
Again, we remark that the operator $\calL_{w}$ is only $\bbR$-linear.
In particular, from $\bfD_{Q}Q=0$, we observe the self-dual factorization
of $i\calL_{Q}$: 
\begin{equation}
i\calL_{Q}=iL_{Q}^{\ast}L_{Q}.\label{eq:calLQ}
\end{equation}
This identity was first observed in \cite{LawrieOhShashahani_unpub}.
Thus, the linearization of \eqref{eq:CSS-m-equiv} at $Q$ is 
\begin{equation}
\rd_{t}\eps+i\calL_{Q}\eps=0,\quad\text{or}\quad\rd_{t}\eps+iL_{Q}^{\ast}L_{Q}\eps=0.\label{eq:lin-CSS-Q}
\end{equation}

Finally, we briefly recall the formal generalized kernel relations
of the linearized operator $i\calL_{Q}$. We have 
\[
N_{g}(i\calL_{Q})=\mathrm{span}_{\bbR}\{\Lambda Q,iQ,\tfrac{i}{4}r^{2}Q,\rho\}
\]
with the relations (see \cite[Proposition 3.4]{KimKwon2019arXiv})
\begin{equation}
\left\{ \begin{aligned}i\calL_{Q}(i\tfrac{r^{2}}{4}Q) & =\Lambda Q, & i\calL_{Q}\rho & =iQ,\\
i\calL_{Q}(\Lambda Q) & =0, & i\calL_{Q}(iQ) & =0,
\end{aligned}
\right.\label{eq:gen-kernel-rel}
\end{equation}
where the existence of $\rho$ is given in \cite[Lemma 3.6]{KimKwon2019arXiv}.
In fact, we have: 
\begin{equation}
\left\{ \begin{aligned}L_{Q}(i\tfrac{r^{2}}{4}Q) & =i\tfrac{r}{2}Q, & L_{Q}\rho & =\tfrac{1}{2(m+1)}rQ,\\
L_{Q}^{\ast}(i\tfrac{r}{2}Q) & =-i\Lambda Q, & L_{Q}^{\ast}(\tfrac{1}{2(m+1)}rQ) & =Q,\\
L_{Q}(\Lambda Q) & =0, & L_{Q}(iQ) & =0.
\end{aligned}
\right.\label{eq:gen-kernel-rel-LQ}
\end{equation}

\subsection{\label{subsec:Adapted-function-spaces}Adapted function spaces}

In this subsection, we quickly recall the equivariant Sobolev spaces
and the adapted function space $\dot{\calH}_{m}^{1}$. For more details,
see \cite{KimKwon2019arXiv,KimKwonOh2020arXiv}.

For $s\geq0$, we denote by $H_{m}^{s}$ and $\dot{H}_{m}^{s}$ the
restriction of the usual Sobolev spaces $H^{s}(\bbR^{2})$ and $\dot{H}^{s}(\bbR^{2})$
on $m$-equivariant functions. For high equivariance indices, we have
the generalized Hardy's inequality \cite[Lemma A.7]{KimKwon2019arXiv}:
whenever $0\leq k\leq|m|$, we have 
\begin{equation}
\||f|_{-k}\|_{L^{2}}\sim\|f\|_{\dot{H}_{m}^{k}}.\label{eq:GenHardySection2}
\end{equation}
Specializing this to $k=1$, we have the \emph{Hardy-Sobolev inequality}
\cite[Lemma A.6]{KimKwon2019arXiv}: whenever $|m|\geq1$, we have
\begin{equation}
\|r^{-1}f\|_{L^{2}}+\|f\|_{L^{\infty}}\lesssim\|f\|_{\dot{H}_{m}^{1}}.\label{eq:HardySobolevSection2}
\end{equation}
Note in general that $H_{0}^{1}\hookrightarrow L^{\infty}$ is \emph{false}.
Finally, we define the weighted Sobolev space $H_{m}^{1,1}$ equipped
with the norm 
\[
\|f\|_{H_{m}^{1,1}}^{2}\coloneqq\|f\|_{H_{m}^{1}}^{2}+\|rf\|_{L^{2}}^{2}.
\]

Next, we define the adapted function space $\dot{\calH}_{m}^{1}$.
This space is motivated by the linear coercivity of energy, namely
the coercivity estimates for the linearized Bogomol'nyi operator $L_{Q}$
at the $\dot{H}^{1}$-level. The available Hardy-type controls on
$f$ from $\|L_{Q}f\|_{L^{2}}$ are different for the cases $m=0$
and $m\geq1$. When $m\geq1$, we have a coercivity of $L_{Q}$ in
terms of the usual $\dot{H}_{m}^{1}$-norm. Thus we let
\[
\dot{\calH}_{m}^{1}\coloneqq\dot{H}_{m}^{1}\quad\text{when }m\geq1.
\]
When $m=0$, the adapted function space $\dot{\calH}_{0}^{1}$ is
defined by the norm 
\[
\|f\|_{\dot{\calH}_{0}^{1}}\coloneqq\|\partial_{r}f\|_{L^{2}}+\|\langle\log_{-}r\rangle^{-1}r^{-1}f\|_{L^{2}}.
\]
We remark that the logarithmic loss near the origin $r=0$ is introduced
due to the failure of Hardy's inequality when $m=0$. Let us note
$\dot{\calH}_{0}^{1}\embed\dot{H}_{0}^{1}$ and $\dot{\calH}_{0}^{1}\cap L^{2}=H_{0}^{1}$.
One also has the following weighted $L^{\infty}$-estimate 
\begin{equation}
\|\langle\log_{-}r\rangle^{-\frac{1}{2}}f\|_{L^{\infty}}\lesssim\|f\|_{\dot{\calH}_{0}^{1}},\label{eq:weighted-Linfty-est}
\end{equation}
which follows from integrating the inequality 
\[
\Big|\rd_{r}\Big(\frac{|f|^{2}}{\langle\log_{-}r\rangle}\Big)\Big|\leq\Big(|\rd_{r}f|+\frac{|f|}{r\langle\log_{-}r\rangle}\Big)\cdot\frac{|f|}{\langle\log_{-}r\rangle}
\]
and applying the fundamental theorem of calculus.

We now state the coercivity estimates of $L_{Q}$ at the $\dot{H}^{1}$-level.
To obtain the coercivity of $L_{Q}$, it is necessary to preclude
the kernel elements $\Lambda Q$ and $iQ$ of $L_{Q}$. We do this
by imposing suitable otrhogonality conditions. We fix profiles $\calZ_{1},\calZ_{2}\in C_{c,m}^{\infty}$
satisfying \emph{the transversality condition} 
\begin{equation}
\det\begin{pmatrix}(\Lambda Q,\calZ_{1})_{r} & (iQ,\calZ_{1})_{r}\\
(\Lambda Q,\calZ_{2})_{r} & (iQ,\calZ_{2})_{r}
\end{pmatrix}\neq0.\label{eq:Z1Z2-transversality}
\end{equation}

\begin{lem}[Coercivity of $L_{Q}$; \cite{KimKwon2019arXiv,KimKwonOh2020arXiv}]
\label{lem:linear-coercivity}Let $m\geq0$. Let $\calZ_{1},\calZ_{2}\in C_{c,m}^{\infty}$
satisfy \eqref{eq:Z1Z2-transversality}. Then, 
\begin{equation}
\|L_{Q}f\|_{L^{2}}\sim\|f\|_{\dot{\calH}_{m}^{1}},\qquad\forall f\in\dot{\calH}_{m}^{1}\text{ with }(f,\mathcal{Z}_{1})_{r}=(f,\mathcal{Z}_{2})_{r}=0.\label{eq:coercivity}
\end{equation}
\end{lem}

\subsection{Duality estimates}

In this subsection, we collect estimates for the nonlinearity of \eqref{eq:CSS-m-equiv}.
Some of the following multilinear estimates already appeared in \cite{KimKwon2019arXiv}.
Here we slightly generalize them for our needs.

We first introduce several more pieces of notation, in order to estimate
systematically the errors from the nonlinearity of \eqref{eq:CSS-m-equiv}.
Denote by $\calN(u)$ the nonlinearity of \eqref{eq:CSS-m-equiv}:
\[
\calN(u)\coloneqq(-|u|^{2}+\tfrac{2m}{r^{2}}A_{\theta}[u]+\tfrac{1}{r^{2}}A_{\theta}^{2}[u]+A_{t}[u])u.
\]
The nonlinearity $\calN(u)$ decomposes into the sum of the cubic
and quintic nonlinearities: 
\[
\calN=\calN_{3,0}+m(\calN_{3,1}+\calN_{3,2})+\calN_{5,1}+\calN_{5,2},
\]
where we abbreviate $\calN_{\ast}(u)\coloneqq\calN_{\ast}(u,\dots,u)$
(where $\ast$ is a place-holder) and denote the cubic nonlinearities
by 
\begin{align*}
\calN_{3,0}(\psi_{1},\psi_{2},\psi_{3}) & \coloneqq-\Re(\br{\psi_{1}}\psi_{2})\psi_{3},\\
\calN_{3,1}(\psi_{1},\psi_{2},\psi_{3}) & \coloneqq\tfrac{2}{r^{2}}A_{\theta}[\psi_{1},\psi_{2}]\psi_{3},\\
\calN_{3,2}(\psi_{1},\psi_{2},\psi_{3}) & \coloneqq-(\tint r{\infty}\Re(\br{\psi_{1}}\psi_{2})\tfrac{dr'}{r'})\psi_{3},
\end{align*}
and the quintic nonlinearities by 
\begin{align*}
\calN_{5,1}(\psi_{1},\dots,\psi_{5}) & \coloneqq\tfrac{1}{r^{2}}A_{\theta}[\psi_{1},\psi_{2}]A_{\theta}[\psi_{3},\psi_{4}]\psi_{5},\\
\calN_{5,2}(\psi_{1},\dots,\psi_{5}) & \coloneqq-(\tint r{\infty}A_{\theta}[\psi_{1},\psi_{2}]\Re(\br{\psi_{3}}\psi_{4})\tfrac{dr'}{r'})\psi_{5}.
\end{align*}
We remark that $\calN_{3,1}$ and $\calN_{3,2}$ do not appear in
the case $m=0$.

In view of the Hamiltonian structure of \eqref{eq:CSS-m-equiv}, the
nonlinearity of \eqref{eq:CSS-m-equiv} arises as a part of the functional
derivative of the energy, i.e., 
\[
\nabla E[u]=-\Delta^{(m)}u+\calN(u).
\]
In order to relate $\calN_{\ast}$ with each component of the energy,
we decompose 
\[
E[u]=\tfrac{1}{2}\tint{}{}(|\partial_{r}u|^{2}+\tfrac{|m|^{2}}{r^{2}}|u|^{2})+\calM_{4,0}[u]+m\mathcal{M}_{4,1}[u]+\mathcal{M}_{6}[u],
\]
where we abbreviate $\calM_{\ast}(u)\coloneqq\calM_{\ast}(u,\dots,u)$
and denote the multilinear forms by 
\begin{align*}
\calM_{4,0}(\psi_{1},\dots,\psi_{4}) & \coloneqq-\tfrac{1}{4}\Re(\br{\psi_{1}}\psi_{2})\Re(\br{\psi_{3}}\psi_{4}),\\
\calM_{4,1}(\psi_{1},\dots,\psi_{4}) & \coloneqq\tint{}{}\tfrac{1}{r^{2}}A_{\theta}[\psi_{1},\psi_{2}]\Re(\br{\psi_{3}}\psi_{4}),\\
\calM_{6}(\psi_{1},\dots,\psi_{6}) & \coloneqq\tfrac{1}{2}\tint{}{}\tfrac{1}{r^{2}}A_{\theta}[\psi_{1},\psi_{2}]A_{\theta}[\psi_{3},\psi_{4}]\Re(\br{\psi_{5}}\psi_{6}).
\end{align*}
It is then easy to verify that 
\begin{equation}
\left\{ \begin{aligned}(\calN_{3,0}(\psi_{1},\psi_{2},\psi_{3}),\psi_{4})_{r} & =4\calM_{4,0}(\psi_{1},\psi_{2},\psi_{3},\psi_{4}),\\
(m\calN_{3,1}(\psi_{1},\psi_{2},\psi_{3}),\psi_{4})_{r} & =2m\calM_{4,1}(\psi_{1},\psi_{2},\psi_{3},\psi_{4}),\\
(m\calN_{3,2}(\psi_{1},\psi_{2},\psi_{3}),\psi_{4})_{r} & =2m\calM_{4,1}(\psi_{3},\psi_{4},\psi_{1},\psi_{2}),\\
(\calN_{5,1}(\psi_{1},\psi_{2},\psi_{3},\psi_{4},\psi_{5}),\psi_{6})_{r} & =2\calM_{6}(\psi_{1},\psi_{2},\psi_{3},\psi_{4},\psi_{5},\psi_{6}),\\
(\calN_{5,2}(\psi_{1},\psi_{2},\psi_{3},\psi_{4},\psi_{5}),\psi_{6})_{r} & =4\calM_{6}(\psi_{1},\psi_{2},\psi_{5},\psi_{6},\psi_{3},\psi_{4}).
\end{aligned}
\right.\label{eq:duality-relations}
\end{equation}
We remark that $\calM_{4,1}$ does not appear in the case $m=0$.

We turn to study the boundedness properties of $\calM_{\ast}$ and
$\calN_{\ast}$. Note that the above relations, in view of duality,
tell us that estimates for the multilinear forms $\calM_{\ast}$ might
transfer to those of $\calN_{\ast}$. We start with the mapping properties
of the integral operators:
\begin{lem}[Mapping properties for integral operators]
Let $1\leq p,q\leq\infty$ and $s\in[0,2]$ be such that $(p,q,s)=(1,\infty,0)$
or $\frac{1}{q}+1=\frac{1}{p}+\frac{s}{2}$ with $p>1$. Then, we
have 
\begin{equation}
\Big\|\frac{1}{r^{s}}\int_{0}^{r}f(r')r'dr'\Big\|_{L^{q}}\lesssim_{p}\|f\|_{L^{p}}.\label{eq:integral-operator-bdd1}
\end{equation}
\end{lem}

\begin{proof}
Note by the definition of $\|f\|_{L^{1}}$ that the estimate is immediate
when $(p,q,s)=(1,\infty,0)$. Henceforth, we assume $p>1$. When $q=p$
and $s=2$, then the proof follows from a change of variables and
Minkowski's inequality:
\begin{align*}
\|\tfrac{1}{r^{2}}\tint 0rf(r')r'dr'\|_{L^{p}} & =\|\tint 01f(ru)udu\|_{L^{p}(rdr)}\\
 & \leq\tint 01\|f(ru)\|_{L^{p}(rdr)}udu=\tint 01u^{1-\frac{2}{p}}\|f\|_{L^{p}}du=\tfrac{p}{2(p-1)}\|f\|_{L^{p}}.
\end{align*}
When $q=\infty$ and $s=2-\frac{2}{p}$, then by Hölder we have 
\[
\|\tfrac{1}{r^{s}}\tint 0rf(r')r'dr'\|_{L^{\infty}}\leq\|f\|_{L^{p}}\cdot\sup_{r\in(0,\infty)}\tfrac{1}{r^{s}}\|\chf_{r'\leq r}\|_{L^{p'}(r'dr')}\lesssim_{p}\|f\|_{L^{p}},
\]
where $\frac{1}{p'}=1-\frac{1}{p}$. For $q\in(p,\infty)$, the estimate
follows from the interpolation: 
\[
\|\tfrac{1}{r^{s}}\tint 0rf(r')r'dr'\|_{L^{q}}\leq\|\tfrac{1}{r^{2}}\tint 0rf(r')r'dr'\|_{L^{p}}^{\theta}\|\tfrac{1}{r^{2-(2/p)}}\tint 0rf(r')r'dr'\|_{L^{\infty}}^{1-\theta}\lesssim_{p}\|f\|_{L^{p}},
\]
where $\theta=\frac{p}{q}\in(0,1)$. This completes the proof of \eqref{eq:integral-operator-bdd1}.
\end{proof}
We then record the Hölder- and weighted $L^{1}$-type estimates for
the multilinear forms $\calM_{\ast}$.
\begin{lem}[Duality estimates (Hölder-type)]
\label{lem:duality-estimates-Holder}The following estimates hold.
\begin{itemize}
\item (For $\mathcal{M}_{4,\ast}$) Let $1\leq p,q\leq\infty$ be such that
$\frac{1}{p}+\frac{1}{q}=1$. Then, we have 
\begin{align*}
|\mathcal{M}_{4,0}(\psi_{1},\psi_{2},\psi_{3},\psi_{4})| & \lesssim\|\psi_{1}\psi_{2}\|_{L^{p}}\|\psi_{3}\psi_{4}\|_{L^{q}},\\
|\mathcal{M}_{4,1}(\psi_{1},\psi_{2},\psi_{3},\psi_{4})| & \lesssim_{p}\|\psi_{1}\psi_{2}\|_{L^{p}}\|\psi_{3}\psi_{4}\|_{L^{q}},\quad\text{if }(p,q)\neq(1,\infty).
\end{align*}
\item (For $\mathcal{M}_{6}$) Let $1\leq p,q,r\leq\infty$ be such that
$\frac{1}{p}+\frac{1}{q}+\frac{1}{r}=2$ and $(p,q,r)\neq(1,1,\infty)$.
Then, we have 
\[
|\mathcal{M}_{6}(\psi_{1},\dots,\psi_{6})|\lesssim_{p,q}\|\psi_{1}\psi_{2}\|_{L^{p}}\|\psi_{3}\psi_{4}\|_{L^{q}}\|\psi_{5}\psi_{6}\|_{L^{r}}.
\]
\end{itemize}
\end{lem}

\begin{proof}
For $\calM_{4,0}$, this is just Hölder's inequality. For $\calM_{4,1}$,
we assume $p>1$ and apply \eqref{eq:integral-operator-bdd1} with
$s=2$ to have 
\[
|\mathcal{M}_{4,1}(\psi_{1},\psi_{2},\psi_{3},\psi_{4})|\lesssim_{p}\|\tfrac{1}{r^{2}}A_{\theta}[\psi_{1},\psi_{2}]\|_{L^{p}}\|\psi_{3}\psi_{4}\|_{L^{q}}\lesssim\|\psi_{1}\psi_{2}\|_{L^{p}}\|\psi_{3}\psi_{4}\|_{L^{q}}.
\]
For $\calM_{6}$, assume $(p,q,r)\neq(1,1,\infty)$. By symmetry,
we may assume $q>1$. We then use \eqref{eq:integral-operator-bdd1}
to have 
\begin{align*}
|\mathcal{M}_{6}(\psi_{1},\dots,\psi_{6})| & \lesssim\|\tfrac{1}{r^{2-(2/p)}}A_{\theta}[\psi_{1},\psi_{2}]\|_{L^{\infty}}\|\tfrac{1}{r^{2/p}}A_{\tht}[\psi_{3},\psi_{4}]\|_{L^{r'}}\|\psi_{5}\psi_{6}\|_{L^{r}}\\
 & \lesssim_{p,q}\|\psi_{1}\psi_{2}\|_{L^{p}}\|\psi_{3}\psi_{4}\|_{L^{q}}\|\psi_{5}\psi_{6}\|_{L^{r}},
\end{align*}
where we denoted $\tfrac{1}{r'}\coloneqq1-\tfrac{1}{r}$. This completes
the proof.
\end{proof}
\begin{lem}[Duality estimates (weighted $L^{1}$-type)]
\label{lem:duality-estimates-weighted-L1}The following estimates
hold.
\begin{itemize}
\item (For $\mathcal{M}_{4,1}$) Let $w_{12},w_{34}:(0,\infty)\to\bbR_{+}$
be decreasing functions such that $w_{12}(r)w_{34}(r)=\frac{1}{r^{2}}$.
Then, we have 
\begin{align*}
|\mathcal{M}_{4,1}(\psi_{1},\psi_{2},\psi_{3},\psi_{4})| & \lesssim\|w_{12}\psi_{1}\psi_{2}\|_{L^{1}}\|w_{34}\psi_{3}\psi_{4}\|_{L^{1}}.
\end{align*}
\item (For $\mathcal{M}_{6}$) Let $w_{12},w_{34},w_{56}:(0,\infty)\to\bbR_{+}$
be decreasing functions such that $w_{12}(r)w_{34}(r)w_{56}(r)=\frac{1}{r^{2}}$.
Then, we have 
\[
|\mathcal{M}_{6}(\psi_{1},\dots,\psi_{6})|\lesssim\|w_{12}\psi_{1}\psi_{2}\|_{L^{1}}\|w_{34}\psi_{3}\psi_{4}\|_{L^{1}}\|w_{56}\psi_{5}\psi_{6}\|_{L^{1}}.
\]
\end{itemize}
\end{lem}

\begin{proof}
Let us only prove the lemma for $\calM_{6}$. We start from writing
\[
\calM_{6}=\tfrac{1}{2}\tint{}{}w_{12}w_{34}w_{56}A_{\theta}[\psi_{1},\psi_{2}]A_{\theta}[\psi_{3},\psi_{4}]\Re(\br{\psi_{5}}\psi_{6}).
\]
Since $w_{12}$ and $w_{34}$ are decreasing, we have 
\[
|w_{12}(r)\tint 0r\Re(\br{\psi_{1}}\psi_{2})r'dr'|\leq\tint 0r|w_{12}\psi_{1}\psi_{2}|r'dr'\leq\|w_{12}\psi_{1}\psi_{2}\|_{L^{1}}
\]
and a similar estimate for $\psi_{3}\psi_{4}$. Thus 
\begin{align*}
|\calM_{6}(\psi_{1},\dots,\psi_{6})| & \leq\tint 0{\infty}\|w_{12}\psi_{1}\psi_{2}\|_{L^{1}}\|w_{34}\psi_{3}\psi_{4}\|_{L^{1}}|w_{56}\psi_{5}\psi_{6}|rdr\\
 & \leq\|w_{12}\psi_{1}\psi_{2}\|_{L^{1}}\|w_{34}\psi_{3}\psi_{4}\|_{L^{1}}\|w_{56}\psi_{5}\psi_{6}\|_{L^{1}}.
\end{align*}
This completes the proof.
\end{proof}
The following two corollaries follow from the duality relations \eqref{eq:duality-relations}
and the above two lemmas.
\begin{cor}[Nonlinear estimates (Hölder-type)]
\label{cor:nonlinear-estimates-holder}For $p\in[1,\infty]$, denote
by $p'$ the Hölder conjugate exponent. The following estimates hold.
\begin{itemize}
\item (For $\mathcal{N}_{3,k}$) For any $1\leq p_{1},\dots,p_{4}\leq\infty$
with $\sum_{j=1}^{4}\frac{1}{p_{j}}=1$ and $\#\{j:p_{j}=\infty\}\leq1$,
we have 
\[
\|\mathcal{N}_{3,k}(\psi_{1},\psi_{2},\psi_{3})\|_{L^{p_{4}'}}\lesssim_{p_{1},p_{2},p_{3}}\|\psi_{1}\|_{L^{p_{1}}}\|\psi_{2}\|_{L^{p_{2}}}\|\psi_{3}\|_{L^{p_{3}}}.
\]
\item (For $\mathcal{N}_{5,k}$) For any $1\leq p_{1},\dots,p_{6}\leq\infty$
with $\sum_{j=1}^{6}\frac{1}{p_{j}}=2$ and $\#\{j:p_{j}=\infty\}\leq1$,
we have 
\[
\|\mathcal{N}_{5,k}(\psi_{1},\dots,\psi_{5})\|_{L^{p_{6}'}}\lesssim_{p_{1},\dots,p_{5}}\prod_{j=1}^{5}\|\psi_{j}\|_{L^{p_{j}}}.
\]
\end{itemize}
\end{cor}

\begin{cor}[Nonlinear estimates (weighted $L^{2}$-type)]
\label{cor:nonlinear-estimates-weightedL2}The following estimates
hold.
\begin{itemize}
\item (For $\calN_{3,1}$ and $\calN_{3,2}$) Let $w_{1},\dots,w_{3}:(0,\infty)\to\bbR_{+}$
be decreasing functions such that $\prod_{j=1}^{3}w_{3}(r)=\frac{1}{r^{2}}$.
Then, for any $k\in\{1,2\}$, we have 
\[
\|\calN_{3,k}(\psi_{1},\psi_{2},\psi_{3})\|_{L^{2}}\lesssim\prod_{j=1}^{3}\|w_{j}\psi_{j}\|_{L^{2}}.
\]
\item (For $\calN_{5,1}$ and $\calN_{5,2}$) Let $w_{1},\dots,w_{5}:(0,\infty)\to\bbR_{+}$
be decreasing functions such that $\prod_{j=1}^{5}w_{j}(r)=\frac{1}{r^{2}}$.
Then, for any $k\in\{1,2\}$, we have 
\[
\|\calN_{5,k}(\psi_{1},\psi_{2},\psi_{3})\|_{L^{2}}\lesssim\prod_{j=1}^{5}\|w_{j}\psi_{j}\|_{L^{2}}.
\]
\end{itemize}
\end{cor}

\section{\label{sec:Proof-of-Theorem-neg-equiv}Proof of Theorem \ref{thm:asymptotic-description}
when $m<0$}

In this short section, we prove Theorem \ref{thm:asymptotic-description}
when $m<0$. In this case, the only scenario for the long-term dynamics
is the scattering. We first show that \eqref{eq:CSS-m-equiv} is \emph{defocusing}
in the sense that the energy is globally coercive:
\begin{lem}[Nonlinear coercivity for $m<0$]
\label{lem:Nonlin-coer-neg-equiv}Let $m<0$. For any $u\in\dot{H}_{m}^{1}$,
we have 
\[
E[u]\sim_{M[u]}\|u\|_{\dot{H}_{m}^{1}}^{2}.
\]
In particular, there is no nontrivial finite energy solution to the
Bogomol'nyi equation \eqref{eq:Bogomol'nyi-eq} for $m<0$.
\end{lem}

\begin{proof}
As the inequality $E[u]\lesssim_{M[u]}\|u\|_{\dot{H}_{m}^{1}}^{2}$
is obvious, we focus on the proof of the reverse inequality $E[u]\gtrsim_{M[u]}\|u\|_{\dot{H}_{m}^{1}}^{2}$.
By density, we may assume that $u$ is an $m$-equivariant Schwartz
function. In particular, $u(0)=0$. We note that 
\begin{align*}
\tint 0{\infty}|\bfD_{u}u|^{2}rdr & =\tint 0{\infty}|\rd_{r}u-\tfrac{m+A_{\theta}[u]}{r}u|^{2}rdr\\
 & =\tint 0{\infty}\{|\rd_{r}u|^{2}+\tfrac{(m+A_{\theta}[u])^{2}}{r^{2}}|u|^{2}-2\Re(\br{\rd_{r}u}\cdot\tfrac{A_{\theta}[u]}{r}u)\}rdr\\
 & \peq-\tint 0{\infty}2\Re(\br{\rd_{r}u}\cdot mu)dr.
\end{align*}
The last term \emph{vanishes}, thanks to integration by parts. Thus,
we have proved
\begin{equation}
\tint 0{\infty}|\bfD_{u}u|^{2}rdr=\tint 0{\infty}\{|\rd_{r}u|^{2}+\tfrac{(m+A_{\theta}[u])^{2}}{r^{2}}|u|^{2}-2\Re(\br{\rd_{r}u}\cdot\tfrac{A_{\theta}[u]}{r}u)\}rdr.\label{eq:NonlinHardy-negequiv1}
\end{equation}

The last term of RHS\eqref{eq:NonlinHardy-negequiv1} will be absorbed
into the sum of the first and second terms. Indeed, since $m<0$ and
$0\leq-A_{\theta}[u](r)\leq\frac{1}{4\pi}M[u]$, there exists $c=c(M[u])>0$
such that 
\[
|A_{\theta}[u]|\leq(1-c)|m+A_{\theta}[u]|.
\]
Thus the last term of RHS\eqref{eq:NonlinHardy-negequiv1} can be
estimated by
\begin{align*}
|\tint{}{}-2\Re(\br{\rd_{r}u}\cdot\tfrac{A_{\theta}[u]}{r}u)| & \leq(1-c)\tint{}{}|2\br{\rd_{r}u}\cdot\tfrac{|m+A_{\theta}[u]|}{r}u|\\
 & \leq(1-c)\tint{}{}\{|\rd_{r}u|^{2}+\tfrac{(m+A_{\theta}[u])^{2}}{r^{2}}|u|^{2}\}rdr.
\end{align*}
Substituting this into \eqref{eq:NonlinHardy-negequiv1}, we have
\[
E[u]=\tfrac{1}{2}\tint{}{}|\bfD_{u}u|^{2}\geq c\tint{}{}\{|\rd_{r}u|^{2}+\tfrac{(m+A_{\theta}[u])^{2}}{r^{2}}|u|^{2}\}\geq c\|u\|_{\dot{H}_{m}^{1}}^{2},
\]
completing the proof.
\end{proof}
As is standard, the coercivity of energy directly implies the scattering
for all $H_{m}^{1,1}$-solutions via the pseudoconformal transform.
\begin{proof}[Proof of Theorem \ref{thm:asymptotic-description} when $m<0$]
Let $m<0$. Suppose that there is a non-scattering maximal $H_{m}^{1,1}$-solution
$u$ to \eqref{eq:CSS-m-equiv}. By the time-reversal and time-translational
symmetry, we may assume that $u$ is defined on $[1,T_{+})$ and is
non-scattering forwards in time, where $T_{+}\in(1,+\infty]$ is the
forward maximal time of existence.

If $T_{+}<+\infty$, then $u$ is a finite-time blow-up solution with
finite energy. The standard blow-up criterion (as a consequence of
the $H_{m}^{1}$-subcritical local well-posedness) says that $\|u(t)\|_{\dot{H}_{m}^{1}}\to\infty$
as $t\to T_{+}$. This is inconsistent with the nonlinear coercivity
(Lemma \ref{lem:Nonlin-coer-neg-equiv}) and the conservation of energy.

If $T_{+}=+\infty$ but $u$ does not scatter, then the standard equivariant
$L^{2}$-Cauchy theory \cite{LiuSmith2016} says that $u$ has infinite
$L_{t,x}^{4}$-norm
\[
\|u\|_{L_{t,x}^{4}([1,+\infty)\times\bbR^{2})}=+\infty.
\]
Let $v\coloneqq\mathcal{C}u$ be the pseudoconformal transform of
$u$ (see \eqref{eq:def-pseudoconf}). Note that $v$ is defined on
the time interval $[-1,0)$ (having well-defined extension past the
time $t=-1$). Moreover, since the pseudoconformal transform preserves
the space $H_{m}^{1,1}$ as well as the $L_{t,x}^{4}$-norm of the
solution, we see that $v$ is a $H_{m}^{1,1}$-solution with 
\[
\|v\|_{L_{t,x}^{4}([-1,0)\times\bbR^{2})}=+\infty,
\]
meaning that $t=0$ is the forward maximal time of existence. In particular,
$v$ blows up at $t=0$. This is impossible due to the previous paragraph.
This completes the proof.
\end{proof}

\section{\label{sec:Proof-of-Theorem-pos-equiv}Proof of Theorem \ref{thm:asymptotic-description}
when $m\protect\geq0$}

In this section, we prove Theorem \ref{thm:asymptotic-description}
when $m\geq0$. As before, we first reduce the proof of Theorem \ref{thm:asymptotic-description}
to the case of finite-time blow-up solutions.
\begin{proof}[Proof of Theorem \ref{thm:asymptotic-description} for global solutions
assuming the finite-time blow-up case.]
\ 

Assume that $u$ is a $H_{m}^{1,1}$-solution on the time interval
$[1,+\infty)$. If $u$ scatters forwards in time, then there is nothing
to prove. Suppose that $u$ does not scatter forwards in time. Similarly
as in the proof for the $m<0$ case (see the previous section), the
pseudoconformal transformed solution $v\coloneqq\mathcal{C}u$ becomes
a $H_{m}^{1,1}$ finite-time blow-up solution that blows up at $t=0$.
According to Theorem \ref{thm:asymptotic-description} for the finite-time
blow-up case, $v$ admits the decomposition 
\[
v(t)-Q_{\lmb(t),\gmm(t)}\to z^{\ast}\text{ in }L^{2}\text{ as }t\to0^{-},
\]
with $\lmb(t),\gmm(t),z^{\ast}$ satisfying the properties stated
in Theorem \ref{thm:asymptotic-description}. Since $\rd_{r}z^{\ast},\frac{1}{r}z^{\ast},rz^{\ast}\in L^{2}$,
we can view $z^{\ast}$ as a radial part of a $H_{-m-2}^{1,1}$ function.
We rewrite the above decomposition as 
\[
v(t)-Q_{\lmb(t),\gmm(t)}-z_{\mathrm{lin}}(t)\to0\text{ in }L^{2}\text{ as }t\to0^{-},
\]
where $z_{\mathrm{lin}}(t)\coloneqq e^{it\Delta^{(-m-2)}}z^{\ast}$.

Inverting the pseudoconformal transform, we have 
\begin{equation}
u(t)-e^{i\frac{r^{2}}{4t}}Q_{\wh{\lmb}(t),\wh{\gmm}(t)}-[\mathcal{C}z_{\mathrm{lin}}](t)\to0\text{ in }L^{2}\text{ as }t\to+\infty,\label{eq:thm-neg-equiv-1}
\end{equation}
where $\wh{\lmb}(t)\coloneqq t\lmb(-1/t)$ and $\wh{\gmm}(t)\coloneqq\gmm(-1/t)$.
In particular, \eqref{eq:thm1-lmb-upper-bound}-\eqref{eq:thm1-lmb-upper-bound-radial}
for $\lmb$ implies \eqref{eq:thm1-lmb-global-bound}-\eqref{eq:thm1-lmb-global-bound-radial}
for $\wh{\lmb}$. Combining the facts $Q\in L^{2}$ and $\wh{\lmb}(t)\lesssim1$
with the DCT, we can replace the pseudoconformal factor $e^{i\frac{r^{2}}{4t}}$
in the display \eqref{eq:thm-neg-equiv-1} by $1$. Finally, since
$z_{\mathrm{lin}}$ is a $H_{-m-2}^{1,1}$-solution to the $(-m-2)$-equivariant
free Schrödinger equation, so is $\mathcal{C}z_{\mathrm{lin}}$. In
other words, $\mathcal{C}z_{\mathrm{lin}}=e^{it\Delta^{(-m-2)}}u^{\ast}$
for some $u^{\ast}\in H_{-m-2}^{1,1}$ and further regularities $\rd_{r}u^{\ast},\frac{1}{r}u^{\ast},ru^{\ast}\in L^{2}$
follow.
\end{proof}
The rest of this section is devoted to the proof of Theorem \ref{thm:asymptotic-description}
for the finite-time blow-up case.

\subsection{Decomposition of small energy solutions}

Let $u$ be a finite-time blow-up solution with finite energy $E$.
By the standard Cauchy theory of \eqref{eq:CSS-m-equiv}, we have
$\|u(t)\|_{\dot{H}_{m}^{1}}\to\infty$ as $t\to T$ with $T$ the
blow-up time of $u$, whereas $E[u(t)]=E$ by conservation of energy.
Renormalizing $u(t)$, i.e., introducing $v(t,r)\coloneqq\wh{\lmb}(t)u(t,\wh{\lmb}(t)r)$
with $\wh{\lmb}(t)=\|Q\|_{\dot{H}_{m}^{1}}/\|u(t)\|_{\dot{H}_{m}^{1}}$,
we have $\|v(t)\|_{\dot{H}_{m}^{1}}=\|Q\|_{\dot{H}_{m}^{1}}$ and
$E[v(t)]\to0$.

Since we know that $E[w]=0$ if and only if $w=0$ or $w$ is a modulated
soliton, it is natural to expect that $v(t)$ is in some sense near
$Q$ (modulo phase rotation). This is done in Lemma \ref{lem:orbital-stability-small-energy}
below in a qualitative way. In practice, we further need to quantify
this proximity to $Q$. So we will fix the decomposition of $u$ into
\[
u=[Q+\eps]_{\lambda,\gamma}
\]
by imposing suitable orthogonality conditions on $\eps$ (Lemma \ref{lem:dec-near-Q}),
and then quantify the smallness of $\eps$ in terms of the energy
$E$ (Lemma \ref{lem:nonlinear-coercivity}). This motivates the following
proposition.
\begin{prop}[Decomposition]
\label{prop:Decomposition}Let $m\geq0$; let $\calZ_{1},\calZ_{2}\in C_{c,m}^{\infty}$
be the profiles as in \eqref{eq:Z1Z2-transversality}. For any $M>1$,
there exist $0<\alpha^{\ast}\ll\eta\ll1$ such that the following
properties hold for all $u\in H_{m}^{1}$ with $\|u\|_{L^{2}}\leq M$
satisfying the small energy condition $\sqrt{E[u]}\leq\alpha^{\ast}\|u\|_{\dot{H}_{m}^{1}}$:
\begin{enumerate}
\item (Decomposition) there exists unique $(\lambda,\gamma)\in\bbR_{+}\times\bbR/2\pi\bbZ$
such that $\eps\in H_{m}^{1}$ defined by the relation 
\[
u=[Q+\eps]_{\lmb,\gmm}
\]
satisfies the orthogonality conditions 
\begin{equation}
(\eps,\calZ_{1})_{r}=(\eps,\calZ_{2})_{r}=0\label{eq:decomp-orthogonality}
\end{equation}
and smallness 
\begin{equation}
\|\eps\|_{\dot{\calH}_{m}^{1}}<\eta.\label{eq:decomp-smallness}
\end{equation}
\item (Estimate for $\lmb$) We have 
\begin{equation}
\Big|\frac{\|u\|_{\dot{H}_{m}^{1}}}{\|Q\|_{\dot{H}_{m}^{1}}}\lmb-1\Big|\lesssim\|\eps\|_{\dot{H}_{m}^{1}}.\label{eq:decomp-lmb-est}
\end{equation}
\item (Improved smallness of $\eps$) We have 
\begin{equation}
\|\eps\|_{\dot{\calH}_{m}^{1}}\sim_{M}\lambda\sqrt{E[u]}.\label{eq:decomp-nonlinear-coercivity}
\end{equation}
\end{enumerate}
\end{prop}

The rest of this subsection is devoted to the proof of Proposition
\ref{prop:Decomposition}. The proof is separated into three lemmas.

Firstly, we show that the smallness of the ratio $\sqrt{E[u]}/\|u\|_{\dot{H}_{m}^{1}}$
implies that $u$ is close to a modulated soliton in the $\dot{\calH}_{m}^{1}$-topology:
\begin{lem}[Orbital stability for small energy solutions]
\label{lem:orbital-stability-small-energy}For any $M>1$ and $\delta>0$,
there exists $\alpha^{\ast}>0$ such that the following holds. Let
$u\in H_{m}^{1}$ be a nonzero profile satisfying $\|u\|_{L^{2}}\leq M$
and the small energy condition $\sqrt{E[u]}\leq\alpha^{\ast}\|u\|_{\dot{H}_{m}^{1}}$.
Then, there exists $\wh{\gmm}\in\bbR/2\pi\bbZ$ such that 
\[
\|e^{-i\wh{\gamma}}\wh{\lmb}u(\wh{\lmb}\cdot)-Q\|_{\dot{\calH}_{m}^{1}}<\delta,
\]
where $\wh{\lmb}\coloneqq\|Q\|_{\dot{H}_{m}^{1}}/\|u\|_{\dot{H}_{m}^{1}}$.
\end{lem}

\begin{proof}
In view of scaling symmetry, we may assume $\wh{\lmb}=1$.

Suppose not. Then, there exist $\eta>0$ and a sequence $\{w_{n}\}_{n\in\bbN}$
in $H_{m}^{1}$ such that 
\[
E[w_{n}]\to0,\quad\|w_{n}\|_{L^{2}}\leq M,\quad\|w_{n}\|_{\dot{H}_{m}^{1}}=\|Q\|_{\dot{H}_{m}^{1}},
\]
and 
\begin{equation}
\|w_{n}-e^{i\gamma}Q\|_{\dot{\calH}_{m}^{1}}\geq\eta\text{ for any }n\in\bbN\text{ and }\gamma\in\bbR/2\pi\bbZ.\label{eq:coercivity-tmp1}
\end{equation}
Passing to a subsequence, and using compact embeddings, we may assume
that 
\begin{align*}
w_{n} & \weakto w_{\infty}\text{ in }H_{m}^{1},\\
w_{n} & \to w_{\infty}\text{ in }L^{p}\text{ for any }p\in(2,\infty),
\end{align*}
for some $w_{\infty}\in H_{m}^{1}$.

We show that $w_{\infty}$ cannot be zero. Indeed, using \cite[Lemma 3.1]{LiLiu2020arXiv}
\[
\|\bfD_{x}w_{n}\|_{L^{2}}^{2}\coloneqq\|\rd_{r}w_{n}\|_{L^{2}}^{2}+\|\tfrac{m+A_{\theta}[w_{n}]}{r}w_{n}\|_{L^{2}}^{2}\sim_{M}\|w_{n}\|_{\dot{H}_{m}^{1}}^{2},
\]
the definition \eqref{eq:energy-Coulomb-form} of energy, and $E[w_{n}]\to0$,
we have 
\begin{align*}
\|w_{\infty}\|_{L^{4}}^{4}=\lim_{n\to\infty}\|w_{n}\|_{L^{4}}^{4} & =\lim_{n\to\infty}(-4E[w_{n}]+2\|\bfD_{x}w_{n}\|_{L^{2}}^{2})\\
 & =\lim_{n\to\infty}2\|\bfD_{x}w_{n}\|_{L^{2}}^{2}\gtrsim_{M}\liminf_{n\to\infty}\|w_{n}\|_{\dot{H}_{m}^{1}}^{2}=\|Q\|_{\dot{H}_{m}^{1}}^{2}.
\end{align*}
Thus $w_{\infty}\neq0$.

We now show that $w_{\infty}=Q_{\lambda,\gamma}$ for some $\lambda\in(0,\infty)$
and $\gamma\in\bbR/2\pi\bbZ$. Indeed, on one hand, $E[w_{n}]\to0$
implies that $\bfD_{w_{n}}w_{n}\to0$ in $L^{2}$. On the other hand,
\cite[Lemma 3.2]{LiLiu2020arXiv} says that $\bfD_{w_{n}}w_{n}\weakto\bfD_{w_{\infty}}w_{\infty}$.
Therefore, we have $\bfD_{w_{\infty}}w_{\infty}=0$, which combined
with $w_{\infty}\neq0$ and the uniqueness of zero energy solutions
implies that $w_{\infty}=Q_{\lambda,\gamma}$ for some $\lambda\in(0,\infty)$
and $\gamma\in\bbR/2\pi\bbZ$.

Next, we show that $w_{n}\to Q_{\lambda,\gamma}$ in $\dot{\calH}_{m}^{1}$.
Let us write $w_{n}=[Q+\td w_{n}]_{\lambda,\gamma}$. Note that $\td w_{n}\weakto0$
in $H_{m}^{1}$ and $\td w_{n}\to0$ in $L^{p}$ for any $p\in(2,\infty)$.
Expanding the expression $E[Q+\td w_{n}]=\lambda^{-2}E[w_{n}]\to0$
using \eqref{eq:linearized-energy-expn} and applying the duality
estimate (Lemma \ref{lem:duality-estimates-Holder}), we see that
\[
\tfrac{1}{2}\|L_{Q}\td w_{n}\|_{L^{2}}^{2}=E[Q+\td w_{n}]+O_{M}(\|\td w_{n}\|_{L^{4}}^{3})\to0.
\]
Combining this with the subcoercivity estimate (see \cite[Lemma A.5]{KimKwon2020arXiv}
for $m\geq1$ and \cite[Lemma A.3]{KimKwonOh2020arXiv} for $m=0$)
and $\td w_{n}\to0$ in $L^{p}$ for some $p\in(2,\infty)$, we conclude
that $\td w_{n}\to0$ in $\dot{\calH}_{m}^{1}$. This shows $w_{n}\to Q_{\lambda,\gamma}$
in $\dot{\calH}_{m}^{1}$.

We are now ready to derive a contradiction. Note that $\lambda=1$
because 
\[
\|Q\|_{\dot{H}_{m}^{1}}=\lim_{n\to\infty}\|w_{n}\|_{\dot{H}_{m}^{1}}=\|Q_{\lambda,\gamma}\|_{\dot{H}_{m}^{1}}=\lambda^{-1}\|Q\|_{\dot{H}_{m}^{1}}.
\]
Thus $w_{n}\to e^{i\gamma}Q$ in $\dot{\calH}_{m}^{1}$, contradicting
\eqref{eq:coercivity-tmp1}. This completes the proof.
\end{proof}
Having established that $u$ is close to a modulated soliton, we fix
the decomposition $u=[Q+\eps]_{\lmb,\gmm}$ by imposing the orthogonality
conditions \eqref{eq:decomp-orthogonality}. In fact, we prove the
following:
\begin{lem}[Decomposition near $Q$]
\label{lem:dec-near-Q}For $\delta>0$, let us denote by $\mathcal{T}_{\delta}$
the set of $u\in\dot{\calH}_{m}^{1}$ satisfying 
\begin{equation}
\inf_{\lmb'\in\bbR_{+},\,\gmm'\in\bbR/2\pi\bbZ}\|u_{(\lmb')^{-1},-\gmm'}-Q\|_{\dot{\calH}_{m}^{1}}<\delta.\label{eq:dec-tmp}
\end{equation}
For any sufficiently small $\eta>0$, there exists $\delta>0$ such
that the following hold for all $u\in\mathcal{T}_{\delta}$:
\begin{enumerate}
\item There exists unique $(\lmb,\gmm)\in\bbR_{+}\times\bbR/2\pi\bbZ$ such
that $u$ admits the decomposition 
\[
u=[Q+\eps]_{\lmb,\gmm}
\]
satisfying the orthogonality conditions \eqref{eq:decomp-orthogonality}
and smallness $\|\eps\|_{\dot{\calH}_{m}^{1}}<\eta$ (which is \eqref{eq:decomp-smallness}).
\item Moreover, the estimate \eqref{eq:decomp-lmb-est} for $\lmb$ holds.
\end{enumerate}
\end{lem}

\begin{proof}
The proof will follow from a standard application of the implicit
function theorem and $L^{2}$-scaling/phase rotation symmetries.

Equip $\bbR_{+}$ with the metric $d_{\bbR_{+}}(\lmb_{1},\lmb_{2})\coloneqq|\log(\frac{\lmb_{1}}{\lmb_{2}})|$;
equip $\bbR/2\pi\bbZ$ with the metric inherited by the standard metric
on $\bbR$. We then equip the parameter space $\bbR_{+}\times\bbR/2\pi\bbZ$
with the product metric, which we denote by $\mathrm{dist}$. Next,
for $u\in\dot{\calH}_{m}^{1}$, $\lambda\in\bbR_{+}$, and $\gamma\in\bbR/2\pi\bbZ$,
we define $\eps=\eps(\lambda,\gamma,u)$ via the relation $u=[Q+\eps]_{\lambda,\gamma}$.

\textbf{Step 1.} Application of the implicit function theorem.

A standard application of the implicit function theorem yields the
following: there exist $0<\delta_{1},\delta_{2}\ll1$ such that, if
$\|u-Q\|_{\dot{\calH}_{m}^{1}}<\delta_{2}$, then there exists unique
$(\lambda,\gmm)$ in the class $\mathrm{dist}((\lmb,\gmm),(1,0))<\delta_{1}$
such that $(\eps,\calZ_{1})_{r}=(\eps,\calZ_{2})_{r}=0$. Moreover,
the Lipschitz bound $\mathrm{dist}((\lmb,\gmm),(1,0))\lesssim\|u-Q\|_{\dot{\calH}_{m}^{1}}$
holds. Note that $\|\eps\|_{\dot{\calH}_{m}^{1}}\lesssim\|u-Q\|_{\dot{\calH}_{m}^{1}}$
also holds in view of the formula of $\eps$.

On the other hand, towards the proof of the global uniqueness of $(\lmb,\gmm)$
in $\bbR_{+}\times\bbR/2\pi\bbZ$, let us prove the following: if
$\eta\ll\delta_{1}$ and $[Q+\eps]_{\lmb,\gmm}=[Q+\eps']_{\lmb',\gmm'}$
for some $\|\eps\|_{\dot{\calH}_{m}^{1}},\|\eps'\|_{\dot{\calH}_{m}^{1}}<\eta$,
then $\mathrm{dist}((\lmb,\gmm),(\lmb',\gmm'))<\delta_{1}$. Indeed,
by the scaling/rotation symmetries and changing the roles of $\lmb,\gmm,\eps$
and $\lmb',\gmm',\eps'$ if necessary, we may assume $\lmb\geq1$,
$\lmb'=1$, and $\gmm'=0$. Then, the identity $Q-Q_{\lmb,\gmm}=\eps_{\lmb,\gmm}-\eps'$
gives $\|Q-Q_{\lmb,\gmm}\|_{\dot{\calH}_{m}^{1}}\leq\|\eps\|_{\dot{\calH}_{m}^{1}}+\|\eps_{\lmb,\gmm}\|_{\dot{\calH}_{m}^{1}}\lesssim\eta\ll\delta_{1}$,
which implies $\mathrm{dist}((\lmb,\gmm),(1,0))<\delta_{1}$ as desired.

\textbf{Step 2.} Completion of the proof.

Let $\eta\ll\delta_{1}$ and $\delta\ll\min\{\eta,\delta_{2}\}$.
By the $L^{2}$-scaling and phase rotation invariances, we may assume
\[
\|u-Q\|_{\dot{\calH}_{m}^{1}}<\delta.
\]

(1) By the first result of Step 1, there exists $(\lmb,\gmm)$ satisfying
the orthogonality conditions \eqref{eq:decomp-orthogonality} and
the Lipschitz bound $\mathrm{dist}((\lmb,\gmm),(1,0))+\|\eps\|_{\dot{\calH}_{m}^{1}}\lesssim\delta\ll\eta$.
For the proof of uniqueness, if there exists $(\lmb',\gmm',\eps')$
satisfying $(\eps',\calZ_{1})_{r}=(\eps',\calZ_{2})_{r}=0$ and $\|\eps'\|_{\dot{\calH}_{m}^{1}}<\eta$,
then the second result of Step 1 says that $\mathrm{dist}((\lmb,\gmm),(\lmb',\gmm'))<\delta_{1}$.
By the local uniqueness result of Step 1, one must have $\lmb=\lmb'$
and $\gmm=\gmm'$.

(2) It now remains to show the estimate \eqref{eq:decomp-lmb-est}
for $\lmb$. Let $\wh{\lmb}\coloneqq\|Q\|_{\dot{H}_{m}^{1}}/\|u\|_{\dot{H}_{m}^{1}}$.
From the identity 
\[
\|Q\|_{\dot{H}_{m}^{1}}=\wh{\lambda}\|w\|_{\dot{H}_{m}^{1}}=\frac{\wh{\lambda}}{\lambda}\|Q+\eps\|_{\dot{H}_{m}^{1}}
\]
and smallness \eqref{eq:decomp-smallness}, we first have $\frac{\wh{\lmb}}{\lmb}\sim1$.
Together with this, the previous display yields the control \eqref{eq:decomp-lmb-est}.
This completes the proof.
\end{proof}
By Lemmas \ref{lem:orbital-stability-small-energy} and \ref{lem:dec-near-Q},
we have proved all the statements of Proposition \ref{prop:Decomposition}
except the improved smallness \eqref{eq:decomp-nonlinear-coercivity}
of $\eps$. We only know that $\|\eps\|_{\dot{\calH}_{m}^{1}}=o_{\alpha^{\ast}\to0}(1)$
so far. In our next lemma, we show that this qualitative smallness
can be improved to the following quantitative smallness:
\begin{lem}[Nonlinear coercivity of energy]
\label{lem:nonlinear-coercivity}For any $M>0$, there exists $\eta>0$
such that the nonlinear coercivity
\begin{equation}
E[Q+\eps]\sim_{M}\|\eps\|_{\dot{\calH}_{m}^{1}}^{2}\label{eq:comparability}
\end{equation}
holds for any $\eps\in H_{m}^{1}$ with $\|\eps\|_{L^{2}}\leq M$
satisfying the orthogonality conditions \eqref{eq:decomp-orthogonality}
and smallness $\|\eps\|_{\dot{\calH}_{m}^{1}}\leq\eta$.
\end{lem}

\begin{rem}
The proof of Lemma \ref{lem:nonlinear-coercivity} not only relies
on the linear coercivity of $L_{Q}$ (Lemma \ref{lem:linear-coercivity}),
but also on a Hardy inequality \eqref{eq:HardyClaim} for the operator
$\bfD_{Q}-\frac{A_{\theta}[\eps]}{r}$. We will call \eqref{eq:HardyClaim}
the \emph{nonlinear Hardy inequality}, because the multiplication
by $\frac{A_{\theta}[\eps]}{r}$ cannot be treated perturbatively
when $\|\eps\|_{L^{2}}$ is allowed to be large. In fact, the proof
of \eqref{eq:HardyClaim} is reminiscent of the proof of the global
coercivity of energy in the negative equivariance case.
\end{rem}

\begin{proof}
In the proof, we will need an additional parameter $R$ satisfying
the parameter dependence $0<\eta\ll R^{-1}\ll M^{-1}$. Thus we may
freely replace the error terms such as $O_{M}(1)\cdot o_{R\to\infty}(1)$
and $O_{M,R}(1)\cdot o_{\eta\to0}(1)$ by $o_{R\to\infty}(1)$ and
$o_{\eta\to0}(1)$, respectively. Moreover, we may bound $o_{\eta\to0}(1)$
by $o_{R\to\infty}(1)$.

We start from writing the energy functional using the self-dual form
\eqref{eq:energy-self-dual-form}: 
\[
2E[Q+\eps]=\|\bfD_{Q+\eps}(Q+\eps)\|_{L^{2}}^{2}=\|L_{Q}\eps-\tfrac{2A_{\theta}[Q,\eps]}{r}\eps-\tfrac{A_{\theta}[\eps]}{r}\eps\|_{L^{2}}^{2}.
\]
Note that the middle term is a perturbative error
\begin{align*}
\|\tfrac{2}{r}A_{\theta}[Q,\eps]\eps\|_{L^{2}} & \lesssim(\tint 0{\infty}Q\langle r\rangle|\eps|dr)\cdot\|\tfrac{1}{\langle r\rangle}\eps\|_{L^{2}}\\
 & \lesssim\|\eps\|_{\dot{\calH}_{m}^{1}}^{\frac{3}{2}}\|\eps\|_{L^{2}}^{\frac{1}{2}}\lesssim O_{M}(1)\cdot o_{\eta\to0}(1)\cdot\|\eps\|_{\dot{\calH}_{m}^{1}}^{2}\lesssim o_{\eta\to0}(\|\eps\|_{\dot{\calH}_{m}^{1}}^{2}),
\end{align*}
whereas the remaining terms are not perturbative: $\|L_{Q}\eps\|\sim\|\eps\|_{\dot{\calH}_{m}^{1}}$
due to \eqref{eq:coercivity} and (using \eqref{eq:integral-operator-bdd1})
\begin{align*}
\|\tfrac{A_{\theta}[\eps]}{r}\eps\|_{L^{2}} & \lesssim\|\tfrac{1}{r}\tint 0r|\eps|^{2}r'dr'\|_{L^{\infty}}\|\eps\|_{L^{2}}\\
 & \lesssim\||\eps|^{2}\|_{L^{2}}\|\eps\|_{L^{2}}\lesssim\|\eps\|_{L^{2}}^{2}\|\eps\|_{\dot{H}_{m}^{1}}\lesssim_{M}\|\eps\|_{\dot{\calH}_{m}^{1}}.
\end{align*}
The above estimates in particular yield the boundedness $E[Q+\eps]\lesssim_{M}\|\eps\|_{\dot{\calH}_{m}^{1}}^{2}$.
For the proof of the reverse inequality, we note 
\begin{equation}
2E[Q+\eps]=\|L_{Q}\eps-\tfrac{A_{\theta}[\eps]}{r}\eps\|_{L^{2}}^{2}+o_{\eta\to0}(\|\eps\|_{\dot{\calH}_{m}^{1}}^{2}).\label{eq:coercivity-tmp2}
\end{equation}
as a consequence of the above estimates.

We then separately consider the coercivity of $\|L_{Q}\eps-\frac{A_{\theta}[\eps]}{r}\eps\|_{L^{2}}$
in the regions $r\lesssim R$ and $r\gtrsim R$, for $R>1$ sufficiently
large. In the region $r\lesssim R$, we notice that the nonlinear
contribution is small: 
\begin{align*}
\|\tfrac{A_{\theta}[\eps]}{r}\chi_{R}\eps\|_{L^{2}} & \lesssim\||\eps|^{2}\|_{L^{2}}\|\chi_{R}\eps\|_{L^{2}}\\
 & \lesssim\|\eps\|_{L^{2}}\|\eps\|_{\dot{H}_{m}^{1}}\cdot R\|\eps\|_{\dot{\calH}_{m}^{1}}\lesssim MR\|\eps\|_{\dot{\calH}_{m}^{1}}^{2}\lesssim o_{\eta\to0}(\|\eps\|_{\dot{\calH}_{m}^{1}}).
\end{align*}
On the other hand, the nonlocal term of $L_{Q}$ is small in the region
$r\gtrsim R$, thanks to the spatial decay of $Q$: 
\[
\|\tfrac{Q}{r}\tint 0r\Re[Q(1-\chi_{R})\eps]r'dr'\|_{L^{2}}=o_{R\to\infty}(\|\epsilon\|_{\dot{\calH}_{m}^{1}}).
\]
Thus we have shown that 
\[
\|L_{Q}\eps-\tfrac{A_{\theta}[\eps]}{r}\eps\|_{L^{2}}=\|L_{Q}(\chi_{R}\eps)+(\bfD_{Q}-\tfrac{A_{\theta}[\eps]}{r})(1-\chi_{R})\eps\|_{L^{2}}+o_{R\to\infty}(\|\eps\|_{\dot{\calH}_{m}^{1}}).
\]
We further observe the following almost orthogonality: 
\begin{align*}
 & \Big|\Big(L_{Q}(\chi_{R}\eps),(\bfD_{Q}-\tfrac{A_{\theta}[\eps]}{r})(1-\chi_{R})\eps\Big)_{r}\Big|\\
 & \lesssim_{M}\|(\chf_{(0,2R]}|\eps|_{-1}+\tfrac{Q}{r}\tint 0r|Q\eps|r'dr')\cdot\chf_{[R,\infty)}|\eps|_{-1}\|_{L^{1}}\\
 & \lesssim_{M}\|\chf_{[R,2R]}|\eps|_{-1}\|_{L^{2}}^{2}+o_{R\to\infty}(\|\eps\|_{\dot{\calH}_{m}^{1}}^{2}).
\end{align*}
As a result, we have arrived at 
\begin{align}
\|L_{Q}\eps-\tfrac{A_{\theta}[\eps]}{r}\eps\|_{L^{2}}^{2} & =\|L_{Q}(\chi_{R}\eps)\|_{L^{2}}^{2}+\|(\bfD_{Q}-\tfrac{A_{\theta}[\eps]}{r})(1-\chi_{R})\eps\|_{L^{2}}^{2}\label{eq:coercivity-tmp4}\\
 & \quad+o_{R\to\infty}(\|\eps\|_{\dot{\calH}_{m}^{1}}^{2})+O_{M}(\|\chf_{[R,2R]}|\eps|_{-1}\|_{L^{2}}^{2}).\nonumber 
\end{align}

Since $R$ is sufficiently large, we may assume that $\calZ_{1}$
and $\calZ_{2}$ are supported in the region $r\leq R$. Thus $\chi_{R}\eps$
satisfies the same orthogonality conditions as $\eps$ and hence the
first term of RHS\eqref{eq:coercivity-tmp4} is coercive by Lemma
\ref{lem:linear-coercivity}: 
\begin{equation}
\|L_{Q}(\chi_{R}\eps)\|_{L^{2}}^{2}\sim\|\chi_{R}\eps\|_{\dot{\calH}_{m}^{1}}^{2}.\label{eq:coercivity-tmp5}
\end{equation}
To deal with the second term of RHS\eqref{eq:coercivity-tmp4}, we
claim the following \emph{nonlinear Hardy inequality}: under the parameter
dependence $R^{-1}\ll M^{-1}$ and $\|\eps\|_{L^{2}}\leq M$, we have,
for $f\in\dot{\calH}_{m}^{1}$ with $f(R)=0$, 
\begin{equation}
\|\chf_{[R,\infty)}(\bfD_{Q}-\tfrac{A_{\theta}[\eps]}{r})f\|_{L^{2}}^{2}\sim_{M}\|\chf_{[R,\infty)}|f|_{-1}\|_{L^{2}}^{2}.\label{eq:HardyClaim}
\end{equation}
Let us assume \eqref{eq:HardyClaim} and finish the proof. Substituting
\eqref{eq:coercivity-tmp5} and \eqref{eq:HardyClaim} into \eqref{eq:coercivity-tmp4},
we have 
\[
\|L_{Q}\eps-\tfrac{A_{\theta}[\eps]}{r}\eps\|_{L^{2}}^{2}\geq c(M)\|\eps\|_{\dot{\calH}_{m}^{1}}^{2}-o_{R\to\infty}(\|\eps\|_{\dot{\calH}_{m}^{1}}^{2})-O_{M}(\|\chf_{[R,2R]}|\eps|_{-1}\|_{L^{2}}^{2})
\]
for some constant $c(M)>0$ depending on $M$. Performing an averaging
argument in $R$ for the last term (that is, one replaces $R$ by
$R'$, takes the integral $\frac{1}{\log R}\int_{R}^{R^{2}}\cdot\frac{dR'}{R'}$,
uses Fubini, and then exploits the smallness $\frac{1}{\log R}=o_{R\to\infty}(1)$)
and applying the parameter dependence $R^{-1}\ll M^{-1}$ yield 
\[
\|L_{Q}\eps-\tfrac{A_{\theta}[\eps]}{r}\eps\|_{L^{2}}^{2}\gtrsim_{M}\|\eps\|_{\dot{\calH}_{m}^{1}}^{2}.
\]
Substituting this into \eqref{eq:coercivity-tmp2} gives the nonlinear
coercivity 
\[
E[Q+\eps]\gtrsim_{M}\|\eps\|_{\dot{\calH}_{m}^{1}}^{2}.
\]
This completes the proof of \eqref{eq:comparability}, assuming the
nonlinear Hardy inequality \eqref{eq:HardyClaim}.

\emph{Proof of the nonlinear Hardy inequality \eqref{eq:HardyClaim}.}

As the $\lesssim_{M}$-inequality is obvious, we only show the $\gtrsim_{M}$-inequality.
Let $f\in\dot{\calH}_{m}^{1}$ be such that $f(R)=0$. We write 
\begin{align*}
 & \tint R{\infty}|(\bfD_{Q}-\tfrac{A_{\theta}[\eps]}{r})f|^{2}rdr\\
 & =\tint R{\infty}\{|\partial_{r}f|^{2}+\tfrac{(m+A_{\theta}[Q]+A_{\theta}[\eps])^{2}}{r^{2}}|f|^{2}-2\Re(\br{\partial_{r}f}\cdot\tfrac{A_{\theta}[\eps]}{r}f)\}rdr\\
 & \quad-\tint R{\infty}2\Re(\br{\partial_{r}f}\cdot(m+A_{\theta}[Q])f)dr
\end{align*}
and integrate by parts the last term: 
\[
-\tint R{\infty}2\Re(\br{\partial_{r}f}\cdot(m+A_{\theta}[Q])f)dr=-\tfrac{1}{2}\tint R{\infty}Q^{2}|f|^{2}rdr,
\]
where we used $f(R)=0$. Thus we have arrived at the main identity
\begin{equation}
\begin{aligned} & \tint R{\infty}|(\bfD_{Q}-\tfrac{A_{\theta}[\eps]}{r})\eps|^{2}rdr\\
 & =\tint R{\infty}\{|\partial_{r}f|^{2}+\tfrac{(m+A_{\theta}[Q]+A_{\theta}[\eps])^{2}}{r^{2}}|f|^{2}-2\Re(\br{\partial_{r}f}\cdot\tfrac{A_{\theta}[\eps]}{r}f)\}rdr\\
 & \quad-\tfrac{1}{2}\tint R{\infty}Q^{2}|f|^{2}rdr.
\end{aligned}
\label{eq:HardyClaim-tmp1}
\end{equation}

Next, we claim that the first term of RHS\eqref{eq:HardyClaim-tmp1}
enjoys a good lower bound: 
\begin{multline}
\tint R{\infty}\{|\partial_{r}f|^{2}+\tfrac{(m+A_{\theta}[Q]+A_{\theta}[\eps])^{2}}{r^{2}}|f|^{2}-2\Re(\br{\partial_{r}f}\cdot\tfrac{A_{\theta}[\eps]}{r}f)\}rdr\\
\gtrsim_{M}\|\chf_{[R,\infty)}|f|_{-1}\|_{L^{2}}^{2}.\label{eq:claim2}
\end{multline}
Indeed, since $m+A_{\theta}[Q]$ and $A_{\theta}[\eps]$ are both
negative (on $[R,\infty)$) and $|A_{\theta}[\eps]|\leq\frac{1}{4\pi}M^{2}$,
we notice that there exists $c=c(M)>0$ satisfying 
\[
\chf_{[R,\infty)}|A_{\theta}[\eps]|\leq(1-c)\chf_{[R,\infty)}|m+A_{\theta}[Q]+A_{\theta}[\eps]|.
\]
Combining this with 
\[
|2\Re(\br{\partial_{r}f}\cdot\tfrac{A_{\theta}[\eps]}{r}f)|\leq(1-c)|\partial_{r}f|^{2}+\tfrac{1}{1-c}|\tfrac{A_{\theta}[\eps]}{r}f|^{2},
\]
we obtain the desired lower bound of the integrand 
\begin{align*}
|\partial_{r}f|^{2}+\tfrac{(m+A_{\theta}[Q]+A_{\theta}[\eps])^{2}}{r^{2}}|f|^{2} & -2\Re(\br{\partial_{r}f}\cdot\tfrac{A_{\theta}[\eps]}{r}f)\\
 & \geq c(|\partial_{r}f|^{2}+\tfrac{(m+A_{\theta}[Q]+A_{\theta}[\eps])^{2}}{r^{2}}|f|^{2})\gtrsim_{M}|f|_{-1}^{2}.
\end{align*}
Integrating this on the region $r\geq R$ yields \eqref{eq:claim2}.

Now the proof of \eqref{eq:HardyClaim} is immediate from the identity
\eqref{eq:HardyClaim-tmp1}, the lower bound \eqref{eq:claim2}, the
estimate 
\[
\tfrac{1}{2}\tint R{\infty}Q^{2}|f|^{2}rdr\lesssim\tfrac{1}{R^{2}}\|\chf_{[R,\infty)}|f|_{-1}\|_{L^{2}}^{2},
\]
and the parameter dependence $R^{-1}\ll M^{-1}$.
\end{proof}
To finish the proof of Proposition \ref{prop:Decomposition}, we recall
that it only suffices to show \eqref{eq:decomp-nonlinear-coercivity}.
This follows from Lemma \ref{lem:nonlinear-coercivity}: 
\[
\|\eps\|_{\dot{\calH}_{m}^{1}}\sim_{M}\sqrt{E[Q+\eps]}=\lambda\sqrt{E[u]}.
\]
This completes the proof of Proposition \ref{prop:Decomposition}.

\subsection{Upper bounds for $\protect\lmb(t)$}

Here and in the next subsection, we let $u$ be a $H_{m}^{1}$-solution
to \eqref{eq:CSS-m-equiv} which blows up forwards in time at $T\in(0,\infty)$.
Let $M$ and $E$ be the mass and energy of $u$, respectively. We
recall that, by the standard Cauchy theory of \eqref{eq:CSS-m-equiv},
$\|u(t)\|_{\dot{H}_{m}^{1}}\to\infty$ as $t\to T$ and thus $\sqrt{E}/\|u(t)\|_{\dot{H}_{m}^{1}}\to0$.
Therefore, for all $t$ sufficiently close to $T$, we can decompose
$u(t)$ according to Proposition \ref{prop:Decomposition}: 
\[
u(t,r)=\frac{e^{i\gamma(t)}}{\lambda(t)}[Q+\eps(t,\cdot)]\Big(\frac{r}{\lambda(t)}\Big)
\]
with the parameters $\lambda(t)$, $\gamma(t)$, and the remainder
$\eps(t)$ satisfying the properties as in Proposition \ref{prop:Decomposition}.
This subsection is devoted to the proofs of the upper bounds \eqref{eq:thm1-lmb-upper-bound}
and \eqref{eq:thm1-lmb-upper-bound-radial}.
\begin{proof}[Proof of \eqref{eq:thm1-lmb-upper-bound}]
Recall the rescaled spacetime variables $(s,y)$ (see \eqref{eq:def-renormalized-var}).
We claim that, as a standard application of the modulation estimate,
for all $t$ sufficiently close to $T$ we have 
\begin{equation}
\Big|\frac{\lambda_{s}}{\lambda}\Big|+|\gamma_{s}|\lesssim\|\eps\|_{\dot{\calH}_{m}^{1}}\lesssim_{M}\lambda\sqrt{E}.\label{eq:thm1-claim1}
\end{equation}
Assuming this claim, we have 
\[
|\lambda_{t}|=\lmb^{-2}|\lmb_{s}|\lesssim_{M}\sqrt{E},
\]
from which the bound $\lambda(t)\lesssim_{M}\sqrt{E}(T-t)$ (which
is \eqref{eq:thm1-lmb-upper-bound}) follows.

Henceforth, we prove the claim \eqref{eq:thm1-claim1}. As the argument
is standard, we will be brief. For a (possibly time-dependent) profile
$\psi$, we note the identity 
\begin{equation}
\begin{aligned}\rd_{s}(\eps,\psi)_{r} & =\tfrac{\lmb_{s}}{\lmb}(\Lambda(Q+\eps),\psi)_{r}-\gmm_{s}(i(Q+\eps),\psi)_{r}\\
 & \quad+(i\calL_{Q}\eps+iR_{Q}(\eps),\psi)_{r}+(\eps,\rd_{s}\psi)_{r}.
\end{aligned}
\label{eq:thm1-claim1-comp}
\end{equation}
If $\psi\in\{\calZ_{1},\calZ_{2}\}$, then we use the orthogonality
conditions \eqref{eq:decomp-orthogonality}, anti-symmetricity of
$\Lambda$ and $i$, the self-dual factorization $\calL_{Q}=L_{Q}^{\ast}L_{Q}$,
and $\partial_{s}\psi=0$ to obtain 
\begin{align*}
\tfrac{\lmb_{s}}{\lmb}\{(\Lambda Q,\calZ_{k})_{r}-(\eps,\Lambda\calZ_{k})_{r}\}+\gmm_{s}\{(iQ,\calZ_{k})_{r}-(\eps,i\calZ_{k})_{r}\}\qquad\\
=(L_{Q}\eps,L_{Q}i\calZ_{k})_{r}+(R_{Q}(\eps),i\calZ_{k})_{r}.
\end{align*}
By the transversality assumption \eqref{eq:Z1Z2-transversality},
$\calZ_{k}\in C_{c,m}^{\infty}$, and $\|\eps(t)\|_{\dot{\calH}_{m}^{1}}\to0$
as $t\to T$, we get 
\begin{equation}
|\tfrac{\lmb_{s}}{\lmb}|+|\gmm_{s}|\lesssim\|L_{Q}\eps\|_{L^{2}}+\|R_{Q}(\eps)\|_{L^{2}}.\label{eq:thm1-claim1-comp2}
\end{equation}
Note that $\|L_{Q}\eps\|_{L^{2}}\lesssim\|\eps\|_{\dot{\calH}_{m}^{1}}$
due to \eqref{eq:coercivity}. In the next paragraph, we show that
\begin{equation}
\|R_{Q}(\eps)\|_{L^{2}}\lesssim_{M}\|\eps\|_{\dot{\calH}_{m}^{1}}^{2}.\label{eq:thm1-claim1-comp3}
\end{equation}
Substituting this into \eqref{eq:thm1-claim1-comp2} completes the
proof of the claim \eqref{eq:thm1-claim1}.

\emph{Proof of \eqref{eq:thm1-claim1-comp3}.} The nonlinear term
$R_{Q}(\eps)$ is a linear combination of $\calN_{\ast}(\psi_{1},\dots,\psi_{\ast})$
where $\#\{j:\psi_{j}=\eps\}\geq2$. If $\#\{j:\psi_{j}=\eps\}\geq3$,
then by Corollary \ref{cor:nonlinear-estimates-holder} we have 
\[
\|\calN_{\ast}(\psi_{1},\dots,\psi_{\ast})\|_{L^{2}}\lesssim(1+\|Q\|_{L^{2}}^{2}+\|\eps\|_{L^{2}}^{2})\|\eps\|_{L^{6}}^{3}\lesssim_{M}\|\eps\|_{L^{6}}^{3}\lesssim_{M}\|\eps\|_{L^{2}}\|\eps\|_{\dot{\calH}_{m}^{1}}^{2}.
\]
If $\#\{j:\psi_{j}=\eps\}=2$, we separately consider the local and
nonlocal nonlinearities; for $\calN_{3,0}$ we use \eqref{eq:weighted-Linfty-est}
to have 
\[
\|\calN_{3,0}(\psi_{1},\psi_{2},\psi_{3})\|_{L^{2}}\lesssim\|Q\eps^{2}\|_{L^{2}}\lesssim\|\langle\log_{-}y\rangle Q\|_{L^{2}}\|\langle\log_{-}y\rangle^{-\frac{1}{2}}\eps\|_{L^{\infty}}^{2}\lesssim\|\eps\|_{\dot{\calH}_{m}^{1}}^{2}
\]
and for the nonlocal nonlinearities we use Corollary \ref{cor:nonlinear-estimates-weightedL2}
to have 
\[
\|\calN_{\ast}(\psi_{1},\dots,\psi_{\ast})\|_{L^{2}}\lesssim_{M}\|\langle\log_{-}y\rangle^{2}Q\|_{L^{2}}\|y^{-1}\langle\log_{-}y\rangle^{-1}\eps\|_{L^{2}}^{2}\lesssim_{M}\|\eps\|_{\dot{\calH}_{m}^{1}}^{2}.
\]
This completes the proof of \eqref{eq:thm1-claim1-comp3}.
\end{proof}
When $m=0$, we can further improve the bound \eqref{eq:thm1-lmb-upper-bound}
to \eqref{eq:thm1-lmb-upper-bound-radial}. The idea is to project
the $\eps$-equation onto the direction of $y^{2}Q$, which is a generalized
kernel element that effectively detects the evolution of $\lmb$.
The logarithmic improvement for the upper bound of $\lmb$ will follow
from the fact that $yQ$ logarithmically fails to lie in $L^{2}$,
which holds only in the $m=0$ case.
\begin{proof}[Proof of \eqref{eq:thm1-lmb-upper-bound-radial} for $m=0$]
Let $\psi=y^{2}Q\chi_{R(t)}$ with $R(t)\coloneqq(T-t)^{-\delta}$
with small $\delta>0$. We note that it suffices to choose any $\delta\in(0,\frac{1}{2})$
in the following analysis. We note the bounds 
\begin{equation}
\begin{aligned}|\tfrac{R_{s}}{R}| & =\lmb^{2}|\tfrac{R_{t}}{R}|\lesssim\tfrac{\lmb^{2}}{T-t}\lesssim_{M}\lmb\sqrt{E},\\
|\rd_{s}\psi| & \lesssim|\tfrac{R_{s}}{R}|\chf_{y\sim R}\lesssim_{M}\lmb\sqrt{E}\chf_{y\sim R}.
\end{aligned}
\label{eq:thm1-tmp5}
\end{equation}

We start by rewriting \eqref{eq:thm1-claim1-comp} as follows: 
\[
\begin{aligned}\tfrac{\lmb_{s}}{\lmb}(\Lambda Q,\psi)_{r}-\rd_{s}(\eps,\psi)_{r} & =\tfrac{\lmb_{s}}{\lmb}(\eps,\Lambda\psi)_{r}-\gmm_{s}(\eps,i\psi)_{r}-(\eps,\rd_{s}\psi)_{r}\\
 & \quad-(i\calL_{Q}\eps,\psi)_{r}-(iR_{Q}(\eps),\psi)_{r}.
\end{aligned}
\]
We further rearrange the LHS of the above display as 
\[
\tfrac{\lmb_{s}}{\lmb}(\Lambda Q,\psi)_{r}-\rd_{s}(\eps,\psi)_{r}=\tfrac{1}{\lmb}\partial_{s}\{\lmb(\Lambda Q,\psi)_{r}-\lmb(\eps,\psi)_{r}\}-(\Lambda Q,\rd_{s}\psi)_{r}+\tfrac{\lmb_{s}}{\lmb}(\eps,\psi)_{r}.
\]
As a result, we have obtained 
\begin{equation}
\begin{aligned}\tfrac{1}{\lmb}\partial_{s}\{\lmb(\Lambda Q,\psi)_{r}-\lmb(\eps,\psi)_{r}\} & =\tfrac{\lmb_{s}}{\lmb}(\eps,y\rd_{y}\psi)_{r}-\gmm_{s}(\eps,i\psi)_{r}-(\eps,\rd_{s}\psi)_{r}\\
 & \quad+(\Lambda Q,\rd_{s}\psi)_{r}-(i\calL_{Q}\eps,\psi)_{r}-(iR_{Q}(\eps),\psi)_{r}.
\end{aligned}
\label{eq:thm1-tmp2}
\end{equation}
In view of 
\begin{align*}
(\Lambda Q,\psi)_{r} & =16\pi\log R+O(1),\\
|(\eps,\psi)_{r}| & \lesssim\|\eps\|_{\dot{\calH}_{0}^{1}}R^{2}\lesssim_{M}\lmb\sqrt{E}(T-t)^{-2\delta}\lesssim_{M,E}(T-t)^{1-2\delta},
\end{align*}
we have 
\begin{equation}
\text{LHS}\eqref{eq:thm1-tmp2}=\frac{1}{\lmb}\partial_{s}\Big\{\lmb\big(16\pi\log R+O(1)\big)\Big\}\label{eq:thm1-tmp3}
\end{equation}

We turn to estimate each term of RHS\eqref{eq:thm1-tmp2}. By \eqref{eq:thm1-claim1}
and \eqref{eq:thm1-tmp5}, we have 
\[
|\tfrac{\lmb_{s}}{\lmb}(\eps,y\rd_{y}\psi)_{r}|+|\gmm_{s}(\eps,i\psi)_{r}|+|(\eps,\rd_{s}\psi)_{r}|\lesssim_{M,E}\lmb\cdot\|\eps\|_{\dot{\calH}_{0}^{1}}R^{2}\lesssim_{M,E}\lmb(T-t)^{1-2\delta}.
\]
Next, by \eqref{eq:thm1-tmp5}, we have 
\[
|(\Lambda Q,\rd_{s}\psi)_{r}|\lesssim_{M}\lmb\sqrt{E}.
\]
Next, the linear term can be bounded as 
\begin{align*}
|(i\calL_{Q}\eps,\psi)_{r}|=|(L_{Q}\eps,L_{Q}i\psi)_{r}| & \lesssim\|L_{Q}\eps\|_{L^{2}}\|L_{Q}i\psi\|_{L^{2}}\\
 & \lesssim_{M}\lmb\sqrt{E}\cdot(\log R)^{\frac{1}{2}}\lesssim_{M}\lmb\sqrt{E}|\log(T-t)|^{\frac{1}{2}},
\end{align*}
Finally, the nonlinear term can be bounded using \eqref{eq:thm1-claim1-comp3}:
\[
|(iR_{Q}(\eps),\psi)_{r}|\lesssim R\|R_{Q}(\eps)\|_{L^{2}}\lesssim_{M}R\|\eps\|_{\dot{\calH}_{m}^{1}}^{2}\lesssim_{M,E}\lmb(T-t)^{1-\delta}.
\]
Therefore, we have obtained the following bound as $t\to T$: 
\begin{equation}
|\text{RHS}\eqref{eq:thm1-tmp2}|\lesssim_{M}\lmb\sqrt{E}|\log(T-t)|^{\frac{1}{2}}.\label{eq:thm1-tmp4}
\end{equation}

Combining \eqref{eq:thm1-tmp3} and \eqref{eq:thm1-tmp4}, we have
as $t\to T$: 
\[
\Big|\frac{1}{\lmb}\partial_{s}\Big\{\lmb\big(16\pi\log R+O(1)\big)\Big\}\Big|\lesssim_{M}\lmb\sqrt{E}|\log(T-t)|^{\frac{1}{2}}.
\]
In terms of the $t$-variable, this reads 
\[
\Big|\partial_{t}\Big\{\lmb\big(16\pi\log R+O(1)\big)\Big\}\Big|\lesssim_{M}\sqrt{E}|\log(T-t)|^{\frac{1}{2}}.
\]
Integrating the above backwards from the blow-up time with $\log R\sim|\log(T-t)|$
and $\lmb\lesssim_{M,E}T-t$ yields 
\[
\lmb\cdot|\log(T-t)|\lesssim_{M}\sqrt{E}(T-t)|\log(T-t)|^{\frac{1}{2}},
\]
which completes the proof of \eqref{eq:thm1-lmb-upper-bound-radial}.
\end{proof}

\subsection{Existence and regularity of asymptotic profile}

In this subsection, we finish the proof of Theorem \ref{thm:asymptotic-description}
by showing that (i) $u(t)$ decomposes as \eqref{eq:thm1-decomp}
for some $z^{\ast}\in L^{2}$, (ii) $z^{\ast}$ enjoys further regularity
$|z^{\ast}|_{-1}\in L^{2}$, and (iii) also $rz^{\ast}\in L^{2}$
if $u$ is a $H_{m}^{1,1}$-solution. We closely follow the argument
of Merle--Raphaël \cite{MerleRaphael2005CMP}.

We first claim the outer $L^{2}$-convergence:
\begin{lem}[Outer $L^{2}$-convergence]
\label{lem:thm1-outer-L2-conv}There exists $z^{\ast}\in L^{2}$
such that for any $R>0$, we have $\chf_{r\geq R}\eps^{\sharp}(t)\to\chf_{r\geq R}z^{\ast}$
in $L^{2}$ as $t\to T$.
\end{lem}

\begin{proof}
In the proof, let $\varphi_{R}$ be a smooth radial cutoff function
satisfying $\varphi_{R}(r)=1$ for $r\geq R$, $\varphi_{R}(r)=0$
for $r\leq\frac{R}{2}$, and $|\varphi_{R}|_{2}\lesssim1$.

We claim that: 
\begin{equation}
\text{For any \ensuremath{R>0}, }\{\varphi_{R}\eps^{\sharp}(t)\}\text{ is Cauchy in \ensuremath{L^{2}} as \ensuremath{t\to T}.}\label{eq:outer-conv-claim}
\end{equation}
Let us finish the proof assuming this claim. For each $R>0$, the
above claim says that $\{\chf_{r\geq R}\eps^{\sharp}(t)\}$ is Cauchy
in $L^{2}$ as $t\to T$. Thus there exists $z_{R}^{\ast}\in L^{2}(r\geq R)$
such that $\chf_{r\geq R}\eps^{\sharp}(t)\to z_{R}^{\ast}$. In view
of the uniqueness of the limit, we have $z_{R_{1}}^{\ast}=\chf_{r\geq R_{1}}z_{R_{2}}^{\ast}$
whenever $R_{1}\geq R_{2}>0$. Therefore, there exists unique profile
$z^{\ast}(r)$ such that, for any $R>0$, $\chf_{r\geq R}z^{\ast}\in L^{2}$
and $\chf_{r\geq R}\eps^{\sharp}(t)\to\chf_{r\geq R}z^{\ast}$ in
$L^{2}$ as $t\to T$. The fact that $z^{\ast}\in L^{2}$ follows
from the uniform boundedness of $\|\eps^{\sharp}(t)\|_{L^{2}}$ with
Fatou's lemma.

We turn to show the claim \eqref{eq:outer-conv-claim}. In this paragraph,
we will reduce the proof of \eqref{eq:outer-conv-claim} to the proof
of \eqref{eq:outer-claim1}. Fix any $\delta_{1}>0$ and $R>0$. It
suffices to show that: there exists $\delta_{2}>0$ such that $\|\varphi_{R}\{\eps^{\sharp}(t)-\eps^{\sharp}(s)\}\|_{L^{2}}<\delta_{1}$
for all $t,s\in(T-\delta_{2},T)$. This will follow from showing that:
there exist $t_{0}<T$ and $\delta_{2}\in(0,T-t_{0})$ such that $\|\varphi_{R}\{\eps^{\sharp}(t+\tau)-\eps^{\sharp}(t)\}\|_{L^{2}}<\delta_{1}$
for all $\tau\in(0,\delta_{2})$ and $t\in[t_{0},T-\tau)$. Now, thanks
to $\varphi_{R}Q_{\lmb(t),\gmm(t)}\to0$ in $L^{2}$ as $t\to T$
(due to $\lmb(t)\to0$), it suffices to show that:
\begin{equation}
\begin{gathered}\text{There exist \ensuremath{t_{0}<T} and \ensuremath{\delta_{2}\in(0,T-t_{0})} such that}\\
\sup_{\tau\in(0,\delta_{2})}\sup_{t\in[t_{0},T-\tau)}\|\varphi_{R}\{u(t+\tau)-u(t)\}\|_{L^{2}}<\delta_{1}.
\end{gathered}
\label{eq:outer-claim1}
\end{equation}

Henceforth, we show \eqref{eq:outer-claim1}. Denote $\td u^{\tau}(t)\coloneqq\varphi_{R}\{u(t+\tau)-u(t)\}$.
Then, 
\begin{align*}
(i\partial_{t}+\Delta_{m})\td u^{\tau}(t) & =[\Delta_{m},\varphi_{R}]u(t+\tau)+\varphi_{R}\mathcal{N}(u(t+\tau))\\
 & \quad-[\Delta_{m},\varphi_{R}]u(t)-\varphi_{R}\mathcal{N}(u(t)),
\end{align*}
so a standard $L^{2}$-energy estimate yields 
\[
\|\td u^{\tau}(t)\|_{L^{2}}\leq\|\td u^{\tau}(t_{0})\|_{L^{2}}+2(T-t_{0})\cdot\sup_{s\in[t_{0},T)}\|[\Delta_{m},\varphi_{R}]u(s)+\varphi_{R}\mathcal{N}(u(s))\|_{L^{2}}.
\]
In the next paragraph, we will show that for any $t_{0}$ sufficiently
close to $T$ 
\begin{equation}
\sup_{s\in[t_{0},T)}\|[\Delta_{m},\varphi_{R}]u(s)+\varphi_{R}\mathcal{N}(u(s))\|_{L^{2}}\lesssim_{R,M,E}1,\label{eq:outer-claim2}
\end{equation}
where $M=M[u]$ and $E=E[u]$. Assuming this, we are led to 
\[
\sup_{\tau\in(0,\delta_{2})}\sup_{t\in[t_{0},T-\tau)}\|\td u^{\tau}(t)\|_{L^{2}}\leq\sup_{\tau\in(0,\delta_{2})}\|\td u^{\tau}(t_{0})\|_{L^{2}}+C(R,M,E)\cdot(T-t_{0}).
\]
Choosing $t_{0}$ sufficiently close to $T$ and then choosing $\delta_{2}>0$
small (using continuity of the flow $t\mapsto u(t)\in L^{2}$ at time
$t=t_{0}$), the claim \eqref{eq:outer-claim1} follows.

It remains to show \eqref{eq:outer-claim2}. We fix a time $s\in[t_{0},T)$
and will obtain estimates uniformly in $s$. For $t_{0}$ sufficiently
close to $T$, we have $\lmb^{-1}(s)R>1$ and the following estimate:
\begin{align*}
 & \|\chf_{r\gtrsim R}\partial_{r}u\|_{L^{2}}+\|\chf_{r\gtrsim R}u\|_{L^{\infty}}\\
 & =\tfrac{1}{\lmb}(\|\chf_{y\gtrsim\lmb^{-1}R}\partial_{y}(Q+\eps)\|_{L^{2}}+\|\chf_{y\gtrsim\lmb^{-1}R}(Q+\eps)\|_{L^{\infty}})\\
 & \lesssim\tfrac{1}{\lmb}((\lmb R^{-1})^{m+2}+\|\eps\|_{\dot{\calH}_{m}^{1}})\lesssim R^{-1}+\lmb^{-1}\|\eps\|_{\dot{\calH}_{m}^{1}}\lesssim_{R,M,E}1.
\end{align*}
Using this, the commutator term in \eqref{eq:outer-claim2} can be
easily treated as 
\[
\|[\Delta_{m},\varphi_{R}]u\|_{L^{2}}\lesssim R^{-1}\|\chf_{r\sim R}\partial_{r}u\|_{L^{2}}+R^{-2}\|\chf_{r\sim R}u\|_{L^{2}}\lesssim_{R,M,E}1.
\]
To estimate the nonlinearity $\mathcal{N}(u)$, we note that 
\[
\mathcal{N}(u)=-|u|^{2}u+\mathcal{N}_{\ast}(u),
\]
where $\mathcal{N}_{\ast}(u)$ consists of the nonlocal nonlinearities
$\calN_{3,1},\calN_{3,2},\calN_{5,1},\calN_{5,2}$. For the local
nonlinearity, we have 
\[
\||u|^{2}\varphi_{R}u\|_{L^{2}}\lesssim\|\chf_{y\gtrsim R}u\|_{L^{\infty}}^{2}\|u\|_{L^{2}}\lesssim_{R}1.
\]
For the nonlocal nonlinearity, by the nonlinear estimate (Corollary
\ref{cor:nonlinear-estimates-weightedL2}), we have 
\[
\|\varphi_{R}\mathcal{N}_{\ast}(u)\|_{L^{2}}\lesssim(1+\|u\|_{L^{2}}^{2})\|u\|_{L^{2}}^{2}\|\tfrac{1}{r^{2}}\varphi_{R}u\|_{L^{2}}\lesssim_{R,M}1.
\]
This ends the proof of \eqref{eq:outer-claim2}.
\end{proof}
Next, we claim the weak $H_{m}^{1}$-convergence:
\begin{lem}[Weak $H_{m}^{1}$-convergence]
\label{lem:thm1-weak-H1-conv}We have $|z^{\ast}|_{-1}\in L^{2}$
and $\eps^{\sharp}(t)\weakto z^{\ast}$ in $H_{m}^{1}$. In particular,
$\eps^{\sharp}(t)\to z^{\ast}$ in $L_{\mathrm{loc}}^{2}$.
\end{lem}

\begin{proof}
First, we show that $\frac{1}{r}z^{\ast}\in L^{2}$. For any $R>0$,
we see from $\lmb(t)\to0$ that 
\begin{align*}
\|\chf_{r\geq R}\tfrac{1}{r}z^{\ast}\|_{L^{2}} & =\lim_{t\to T}\|\chf_{r\geq R}\tfrac{1}{r}\eps^{\sharp}(t)\|_{L^{2}}\\
 & =\lim_{t\to T}\tfrac{1}{\lmb}\|\chf_{y\geq\lmb^{-1}R}\tfrac{1}{y}\eps(t,y)\|_{L^{2}}\leq\limsup_{t\to T}\tfrac{1}{\lmb}\|\eps(t)\|_{\dot{\calH}_{m}^{1}}\lesssim1,
\end{align*}
where the implicit constant is uniform in $R$. Letting $R\to0$,
we have $\frac{1}{r}z^{\ast}\in L^{2}$.

Next, we show that $z^{\ast}\in H_{m}^{1}$ and $\eps^{\sharp}(t)\weakto z^{\ast}$
in $H_{m}^{1}$. By a further subsequence argument, it suffices to
show that (i) $z^{\ast}\in H_{m}^{1}$ and (ii) for any sequence $t_{n}\to T$
there exists a further subsequence $t_{n'}\to T$ such that $\eps^{\sharp}(t_{n'})\weakto z^{\ast}$
in $H_{m}^{1}$. Let $t_{n}\to T$ be arbitrary. Since $\{\eps^{\sharp}(t_{n})\}$
is $H_{m}^{1}$-bounded, it has a further subsequence $\{\eps^{\sharp}(t_{n'})\}$
such that $\eps^{\sharp}(t_{n'})\weakto z_{w}^{\ast}$ for some $z_{w}^{\ast}\in H_{m}^{1}$.
It now suffices to show that $z_{w}^{\ast}=z^{\ast}$. For this, it
suffices show that $\chf_{[R^{-1},R]}z_{w}^{\ast}=\chf_{[R^{-1},R]}z^{\ast}$
for any $R>1$. Fix $R>1$. On one hand, by the Rellich--Kondrachov
theorem, we have $\chf_{[R^{-1},R]}\eps^{\sharp}(t_{n'})\to\chf_{[R^{-1},R]}z_{w}^{\ast}$
in $L^{2}$. On the other hand, by the outer $L^{2}$-convergence
(Lemma \ref{lem:thm1-outer-L2-conv}), we have $\chf_{[R^{-1},R]}\eps^{\sharp}(t_{n'})\to\chf_{[R^{-1},R]}z^{\ast}$
in $L^{2}$. Thus $\chf_{[R^{-1},R]}z_{w}^{\ast}=\chf_{[R^{-1},R]}z^{\ast}$.
This completes the proof of the weak $H_{m}^{1}$-convergence.
\end{proof}
Lemmas \ref{lem:thm1-outer-L2-conv} and \ref{lem:thm1-weak-H1-conv}
show that $z^{\ast}\in L^{2}$, $|z^{\ast}|_{-1}\in L^{2}$, and $\eps^{\sharp}(t)\to z^{\ast}$
in $L^{2}$ as $t\to T$. This proves the decomposition \eqref{eq:thm1-decomp}
and the regularity of $z^{\ast}$ in Theorem \ref{thm:asymptotic-description}
for $H_{m}^{1}$-solutions. If in addition $u$ is a $H_{m}^{1,1}$-solution,
then we can appeal to the virial identities \eqref{eq:virial-1}-\eqref{eq:virial-2}
and observe that $\|ru(t)\|_{L^{2}}$ is bounded as $t\to T$. Thus
by the outer convergence (Lemma \ref{lem:thm1-outer-L2-conv}) and
Fatou property we have 
\begin{align*}
\|rz^{\ast}\|_{L^{2}}=\lim_{R\to0}\|\chf_{r\geq R}rz^{\ast}\|_{L^{2}} & \leq\lim_{R\to0}\liminf_{t\to T}\|\chf_{r\geq R}ru(t)\|_{L^{2}}\\
 & \qquad\leq\limsup_{t\to T}\|ru(t)\|_{L^{2}}<+\infty.
\end{align*}
Hence $rz^{\ast}\in L^{2}$. This ends the proof of Theorem \ref{thm:asymptotic-description}.

\bibliographystyle{plain}
\bibliography{References}

\end{document}